\documentclass[11pt,leqno]{article}
\usepackage{amsmath, amscd, amsthm, amssymb, graphics, xypic, mathrsfs, setspace, fancyhdr, times, enumitem}
\xyoption{all}
\entrymodifiers={+!!<0pt,\fontdimen22\textfont2>}
\DeclareFontFamily{OT1}{pzc}{}
\DeclareFontShape{OT1}{pzc}{m}{it}%
             {<-> s * [1.195] pzcmi7t}{}
\DeclareMathAlphabet{\mathscr}{OT1}{pzc}%
                                 {m}{it}
\usepackage[colorlinks=true,pagebackref=true]{hyperref} 
\hypersetup{backref}

\newcommand{\colim}{\operatorname{colim}}
\newcommand{\Spec}{\operatorname{Spec}}

\newcommand{\isomto}{{\stackrel{\sim}{\;\longrightarrow\;}}}
\newcommand{\isomt}{{\stackrel{{\scriptscriptstyle{\sim}}}{\;\rightarrow\;}}}

\newcommand{\dprime}{{\prime\prime}}

\renewcommand{\hom}{\operatorname{Hom}}

\newcommand{\real}{{\mathbb R}}

\newcommand{\Z}{{\mathbb Z}}
\newcommand{\N}{{\mathbb N}}
\newcommand{\A}{{\mathbb A}}
\newcommand{\aone}{{\mathbb A}^1}
\newcommand{\pone}{{\mathbb P}^1}

\newcommand{\sma}{{\scriptstyle{\wedge}}}

\newcommand{\ext}{\operatorname{Ext}}

\newcommand{\gm}{{{\mathbf G}_{m}}}
\renewcommand{\L}{{\mathcal L}}

\newcommand{\et}{\text{\'et}}
\newcommand{\ho}[1]{\mathscr{H}({#1})}
\newcommand{\hop}[1]{\mathscr{H}_{\bullet}({#1})}

\newcommand{\bpi}{\boldsymbol{\pi}}
\newcommand{\piaone}{{\bpi}^{\aone}}

\newcommand{\Nis}{\operatorname{Nis}}

\newcommand{\Sm}{\mathscr{Sm}}

\newcommand{\Spc}{\mathscr{Spc}}

\newcommand{\K}{{{\mathbf K}}}

\newcommand{\F}{{\mathcal F}}

\newcommand{\Addresses}{{
  \bigskip
  \footnotesize

  A.~Asok, \textsc{Department of Mathematics, University of Southern California, 3620 S. Vermont Ave. KAP 104,
    Los Angeles, CA 90089-2532, United States;} \textit{E-mail address:} \url{asok@usc.edu}

  \medskip

  J.~Fasel, \textsc{Institut Fourier - UMR 5582, Universit\'e Grenoble Alpes, CS 40700, F-38058 Grenoble Cedex 9;} \textit{E-mail address:} \url{jean.fasel@gmail.com}

}}

\newcounter{intro}
\theoremstyle{plain}
\newtheorem{thm}{Theorem}[subsection]

\newtheorem{lem}[thm]{Lemma}
\newtheorem{cor}[thm]{Corollary}
\newtheorem{prop}[thm]{Proposition}
\newtheorem*{claim*}{Claim}  

\newtheorem*{thm*}{Theorem}
\newtheorem*{problem*}{Problem}

\newtheorem{thmintro}{Theorem}

\newtheorem{corintro}[thmintro]{Corollary}

\theoremstyle{definition}
\newtheorem{defn}[thm]{Definition}

\theoremstyle{remark}
\newtheorem{rem}[thm]{Remark}
\newtheorem{remintro}[thmintro]{Remark}

\newtheorem{ex}[thm]{Example}

\numberwithin{equation}{subsection}

\begin{document}
\pagestyle{fancy}
\renewcommand{\sectionmark}[1]{\markright{\thesection\ #1}}
\fancyhead{}
\fancyhead[LO,R]{\bfseries\footnotesize\thepage}
\fancyhead[LE]{\bfseries\footnotesize\rightmark}
\fancyhead[RO]{\bfseries\footnotesize\rightmark}
\chead[]{}
\cfoot[]{}
\setlength{\headheight}{1cm}

\author{Aravind Asok\thanks{Aravind Asok was partially supported by National Science Foundation Awards DMS-0966589 and DMS-1254892.} \and Jean Fasel\thanks{Jean Fasel was partially supported by the DFG Grant SFB Transregio 45.}}

\title{{\bf An explicit $\mathbf{KO}$-degree map and applications}}

\date{}
\maketitle

\begin{abstract}
The goal of this note is to study the analog in unstable $\aone$-homotopy theory of the unit map from the motivic sphere spectrum to the Hermitian K-theory spectrum, i.e., the degree map in Hermitian K-theory.  We show that ``Suslin matrices", which are explicit maps from odd dimensional split smooth affine quadrics to geometric models of the spaces appearing in Bott periodicity in Hermitian K-theory, stabilize in a suitable sense to the unit map.  As applications, we deduce that $K^{MW}_i(F) = GW^i_i(F)$ for $i \leq 3$, which can be thought of as an extension of Matsumoto's celebrated theorem describing $K_2$ of a field.  These results provide the first step in a program aimed at computing the sheaf $\bpi_{n}^{\aone}({\mathbb A}^n \setminus 0)$ for $n \geq 4$.
\end{abstract}

\begin{footnotesize}
\setcounter{tocdepth}{1}
\tableofcontents
\end{footnotesize}

\section{Introduction}
This paper continues the analysis of the $\aone$-homotopy sheaf $\bpi_n^{\aone}({\mathbb A}^n \setminus 0)$ that was initiated in \cite{AsokFaselThreefolds} and \cite{AsokFaselpi3a3minus0}.  An understanding of this sheaf has applications to splitting problems for algebraic vector bundles on smooth affine varieties ``below the stable range" as explained in \cite{AsokFaselpi3a3minus0} and to construction of secondary characteristic classes for algebraic vector bundles on smooth schemes going beyond those discussed in \cite{AsokFaselSecondary}.

In \cite{AsokFaselThreefolds}, we studied $\bpi_2^{\aone}({\mathbb A}^2 \setminus 0)$ by using the $\aone$-weak equivalence $Sp_2 \to {\mathbb A}^2 \setminus 0$ given by ``projection onto the first column" and then studying the $\aone$-homotopy fiber of the stabilization map $Sp_2 \to Sp_{\infty}$.  In \cite{AsokFaselpi3a3minus0}, we studied $\bpi_3^{\aone}({\mathbb A}^3 \setminus 0)$ by observing that there is an $\aone$-weak equivalence $SL_4/Sp_4 \to {\mathbb A}^3 \setminus 0$ and studying the $\aone$-homotopy fiber of the stabilization map $SL_4/Sp_4 \to GL_4/Sp_4 \to GL/Sp$.  As we shall see here, the ``stabilization" map can, in each case, be interpreted as sending an element of an $\aone$-homotopy group to its ``degree" in higher Grothendieck--Witt groups (a generalization of Hermitian K-theory).  In this paper, we study this ``degree" homomorphism;  some of the results presented here were announced in \cite{AsokFaselOberwolfach}.

To motivate our constructions, recall that classical Bott periodicity gives homotopy equivalences $\Omega^{i+8} (\Z \times BO) \cong \Omega^{i} (\Z \times BO)$ and provides explicit geometric models of the various loop spaces of $\Omega^{i} (\Z \times BO)$: in particular, if $i = 1,3,5$ or $7$, then $\Omega^i(\Z \times BO)$ is homotopy equivalent to $O, U/Sp, Sp$ or $U/O$ \cite{Bott}.  The space $\Z \times BO$ is the representing space for (orthogonal, topological) K-theory, and $[pt,\Z \times BO] \cong KO_0^{top}(pt) \cong \Z$.  By adjunction, a generator of $KO_0^{top}(pt)$ determines maps
\[
S^{2i-1} \longrightarrow \Omega^{-2i+1} (\Z \times BO)
\]
for each $i> 0$.

The situation in algebraic geometry is, in outline, similar (details are deferred to the body of the text).  Suppose $k$ is a field having characteristic unequal to $2$.  Write $\hop{k}$ for the pointed $\aone$-homotopy category over $k$ as constructed by Morel--Voevodsky \cite{MV}.  The role of (orthogonal, topological) K-theory is played by the higher Grothendieck--Witt theory; for a smooth scheme $X$ one considers the doubly indexed groups $GW^{j}_i(X)$.  The analog of Bott periodicity in the topological setting is a $4$-periodicity in the index $j$.  The groups $GW^{j}_i(X)$ are known to be representable in the unstable pointed $\aone$-homotopy category by suitable ``loop spaces" of a space $\Z \times B_{\et}O$ by the recent work of Schlichting--Tripathi \cite{SchlichtingTripathi}; here $B_{\et}{O}$ is the space defined in \cite[p. 130]{MV}.  In particular, Schlichting and Tripathi prove that
\[
GW^{-j}_{i+1}(X) \cong \hom_{\hop{k}}(\Sigma^i_s X_+,\Omega^j_{\pone}\Omega^1_s (\Z \times B_{\et}O))
\]
and show that one can construct ind-algebraic varieties $O, GL/Sp, Sp$ and $GL/O$, modeled by ``stable matrices" of an appropriate sort (orthogonal, anti-symmetric, symplectic, or symmetric) $\aone$-weakly equivalent to $\Omega^j_{\pone}\Omega^1_s (\Z \times B_{\et}O)$ for $j = 1,2,3$ or $4$.

For any field $k$ (again having characteristic unequal to $2$), $GW^0_0(k)$ is the Grothendieck--Witt group of $k$, i.e., the Grothendieck group of isomorphism classes of symmetric bilinear forms over $k$.  There is an element $\langle 1 \rangle$ corresponding to the isomorphism class of the symmetric bilinear form $(x,y)\mapsto xy$.  As above, by adjunction, this element corresponds to an element
\[
\Psi_j \in \hom_{\hop{k}}(\Sigma^{j-1}_s \gm^{\sma j},\Omega^j_{\pone}\Omega^1_s (\Z \times B_{\et}O)).
\]
Now, $\Sigma^{j-1}_s \gm^{\sma j}$ is $\aone$-weakly equivalent to ${\mathbb A}^{j} \setminus 0$.  If $Q_{2j-1}$ is the smooth affine quadric hypersurface defined by $\sum_i x_i y_i = 1$, then projection onto the $x_i$'s determines an $\aone$-weak equivalence $Q_{2j-1} \to {\mathbb A}^{j} \setminus 0$.  Our first main result, which is an analog for higher Grothendieck--Witt groups of a theorem of Suslin for algebraic K-theory \cite[Theorem 2.3]{Suslin82b}, shows that the morphism just mentioned is given by an explicit matrix.

\begin{thmintro}[See Theorem \ref{thm:suslinmatrixgwgenerator}]
\label{thmintro:matrixgenerator}
For every integer $n > 0$, the generator $\Psi_n$ is represented by the class of an inductively defined element of $GL(Q_{2n-1})$ that is either orthogonal ($n \equiv 0 \mod 4$), symmetric ($n \equiv 1 \mod 4$), symplectic ($n \equiv 2 \mod 4$), or anti-symmetric ($n \equiv 3 \mod 4$).
\end{thmintro}

By construction, the explicit class $\Psi_n$ induces a morphism
\[
\Psi_{n,i,j}: \bpi_{i,j}^{\aone}({\mathbb A}^n \setminus 0) \longrightarrow \bpi_{i,j}^{\aone}(\Omega^n_{\pone}\Omega^1_s (\Z \times B_{\et}O)).
\]
To simplify notation, we set $\Psi_{n,i} := \Psi_{n,i,0}$.  The map $\Psi_{n,i}$ is necessarily trivial for $i \leq n-2$ because ${\mathbb A}^n \setminus 0$ is $\aone$-$(n-2)$-connected \cite[Theorem 6.38]{MField}, and we now discuss the structure of this homomorphism in degrees $i = n-1$ and $i = n$.

If $K^M_*(F)$ is the Milnor K-theory ring of $F$ and $K^Q_*(F)$ is the Quillen K-theory ring of $F$, there is a unique ring homomorphism, sometimes called the ``natural" homomorphism,
\[
K^M_*(F) \longrightarrow K^Q_*(F)
\]
that induces the identity in degree $\leq 1$.  It follows from Matsumoto's theorem presenting $K^Q_2(F)$ (see, e.g., \cite[Chapter III Theorem 6.1]{KBook}) that the above map is actually an isomorphism in degree $\leq 2$.  Using the matrices $\Psi_n$, Suslin was able to show \cite[Theorem 5.2]{Suslin82b} that a certain composite map $K^M_n(F) \to K^Q_n(F) \to K^M_n(F)$ was multiplication by $(n-1)!$.

Morel showed that, for any integer $n \geq 2$, $\bpi_{n-1}^{\aone}({\mathbb A}^n \setminus 0) \cong \K^{MW}_n$ (both sides are Nisnevich sheaves of abelian groups).  The sections of $\K^{MW}_n$ over finitely generated extensions of the base field are given in terms of explicit generators and relations (see, e.g., \cite[Chapter 3]{MField}). The relationship between Milnor-Witt K-theory and Hermitian K-theory is in many ways analogous to the relationship between Milnor K-theory and Quillen K-theory described in the previous paragraph.  In particular, there is a ``natural" homomorphism of graded rings (defined using generators and relations)
\[
K^{MW}_*(F) \longrightarrow GW^*_*(F)
\]
that is an isomorphism in degree $\leq 1$ by construction and isomorphism in degree $2$ by a result of Suslin \cite[\S 6]{Suslin87}.  We give here a new and elementary construction of this ``natural" homomorphism (and no construction appears in the literature).  As a first step toward extending Suslin's result \cite[Theorem 5.2]{Suslin82b} to Hermitian K-theory, we establish the following result, which also provides an extension of Matsumoto's theorem.

\begin{thmintro}
\label{thmintro:suslinmatricesandmatsumoto}
If $n\geq 2$, the homomorphism $\Psi_{n,n-1}$ coincides with the degree $n$ graded component of the natural homomorphism upon taking sections over finitely generated extensions of the base field.  Moreover, the induced map $K^{MW}_3(F) \to GW^3_3(F)$ is an isomorphism for any field $F$ having characteristic unequal to $2$.
\end{thmintro}

By analyzing the so-called Karoubi periodicity sequences, the above result yields a novel description of the third Hermitian K-theory group of a field having characteristic unequal to $2$.  Recall that $K_3^{ind}(F)$ is defined as the cokernel of the natural homomorphism $K^M_3(F) \to K^Q_3(F)$.

\begin{corintro}
\label{corintro:matsumoto}
If $F$ is any field having characteristic unequal to $2$, then the hyperbolic map $H: K^Q_3(F) \to KO_3(F)$ factors through an isomorphism $K_3^{ind}(F)\cong KO_3(F)$.
\end{corintro}

\begin{remintro}
The natural homomorphism can be identified as an edge map in the motivic spectral sequence via the use of the multiplicative structure in the spectral sequence \cite[Proposition 3.3]{GeisserLevine}, or as an edge map in the Brown--Gersten spectral sequence.  Analogously, the ``natural" homomorphism above arises as an edge map in the Gersten--Grothendieck--Witt spectral sequence.
\end{remintro}

The construction just mentioned is ``compatible with $\pone$-suspension" in the following sense:  the operation of $\pone$-suspension defines a homomorphism $\bpi_{i,j}^{\aone}({\mathbb A}^n \setminus 0) \to \bpi_{i+1,j+1}^{\aone}({\mathbb A}^{n+1} \setminus 0)$ that, for any integer $n \geq 2$, fits into a commutative triangle of the form:
\[
\xymatrix{
\bpi_{i,j}^{\aone}({\mathbb A}^n \setminus 0) \ar[r]\ar[dr]_{\Psi_{n,i,j}} & \bpi_{i+1,j+1}^{\aone}({\mathbb A}^{n+1} \setminus 0) \ar[d]^{\Psi_{n+1,i+1,j+1}}\\
& \mathbf{GW}^n_{i+1}.
}
\]
It is this observation that explains the sense in which the maps $\Psi_n$, which are defined ``unstably", stabilize to the ``degree map" in $KO$-theory, i.e., the map induced by unit map from the motivic sphere spectrum to the Hermitian K-theory spectrum.

\begin{thmintro}
\label{thmintro:degreemapindegreen}
The morphism $\Psi_{n,n,j}$ is an epimorphism for $j \geq n-3$ or $j = 0$ and $n \leq 3$, i.e., the homomorphism
\[
\bpi_n^{\aone}({\mathbb A}^n \setminus 0) \longrightarrow \mathbf{GW}^n_{n+1}
\]
becomes an epimorphism after $(n-3)$-fold contraction.
\end{thmintro}

If $n \geq 4$, the map $\bpi_n^{\aone}({\mathbb A}^n \setminus 0) \to \mathbf{GW}^n_{n+1}$ is not known to be surjective, and describing the image of this homomorphism seems difficult.  However, by Morel's $\aone$-Freudenthal suspension theorem \cite[Theorem 6.61]{MField}, for $n \geq 4$, there are isomorphisms $\bpi_n^{\aone}({\mathbb A}^n \setminus 0) \to \bpi_{n+1}^{\aone}({\pone}^{\sma n})$ and, by definition,
\[
\bpi_{1,j}^{s\aone}({\mathbf S}^0_k) \:= \colim_n \bpi_{n+1,n+j}({\pone}^{\sma n}),
\]
where the groups on the left are the graded components of the first stable $\aone$-homotopy sheaf of the motivic sphere spectrum.  Theorem \ref{thmintro:degreemapindegreen} immediately implies the following result on these stable groups.

\begin{corintro}
\label{corintro:stabilization}
For $j \leq 3$, there are epimorphisms $\bpi_{1,j}^{s\aone}({\mathbf S}^0_k) \to \mathbf{GW}^{-j}_{1-j}$.
\end{corintro}

We close this introduction with some general comments on the structure of the sheaf $\bpi_n^{\aone}({\mathbb A}^n \setminus 0)$ for $n \geq 4$.  In this regard, the paper \cite{AWW} is a companion to this one as it can be viewed as studying the homotopy fiber of $\Psi_n$.  In particular, \cite[Conjecture 5.2.10]{AWW} contains a precise conjecture regarding the structure of $\bpi_n^{\aone}({\mathbb A}^n \setminus 0)$.  This conjecture is important for two completely unrelated reasons.

First, in private correspondence from around 2005, F. Morel gave a conjectural description of the stable motivic $\bpi_1$ sheaf of the sphere spectrum as an extension of $\mathbf{GW}^0_1$ by $\K^M_2/24$.  Corollary \ref{corintro:stabilization} confirms, in particular, that the stable $\bpi_1$ of the sphere spectrum surjects onto $\mathbf{GW}^0_1$.  K. Ormsby and P.-A. {\O}stv{\ae}r \cite{OrmsbyOstvaer} have established Morel's conjecture after taking sections over fields $k$ having cohomological dimension $\leq 2$. Recently, O.R\"ondigs, M. Spitzweck and P.-A. {\O}stv{\ae}r have established Morel's conjecture for fields having characteristic $0$ \cite{RondigsSpitzweckOstvaer}; their results are compatible with ours.  The validity of our conjecture on the structure of $\bpi_n^{\aone}({\mathbb A}^n \setminus 0)$ would immediately imply Morel's conjecture at the level of sheaves.  Second, a positive solution to \cite[Conjecture 5.2.10]{AWW} would yield a complete solution to Murthy's splitting conjecture for projective modules.  We refer the reader to \cite{AsokFaselpi3a3minus0} for more details about Murthy's conjecture.

\subsubsection*{Acknowledgements}
The authors would like to thank Marco Schlichting and Girja Tripathi for providing us with preliminary versions of their work on geometric representation of Hermitian K-theory \cite{SchlichtingTripathi} and for various discussions around the ``natural" homomorphism.  The first author would also like to thank K. Ormsby for discussions on \cite{OrmsbyOstvaer} and Oliver R\"ondigs for discussion on \cite{RondigsSpitzweckOstvaer}.  Finally, we thank the referee for some suggestions which helped improve the exposition.

\subsubsection*{Notation/Conventions}
Throughout the paper $k$ will be a field having characteristic unequal to $2$.  We write $\Sm_k$ for the category of schemes separated, smooth and of finite type over $k$.  We write $\Spc_k$ for the category of simplicial Nisnevich sheaves on $\Sm_k$.  We write $\ho{k}$ or $\hop{k}$ for the Morel--Voevodsky $\aone$-homotopy category or its pointed version.  If $\mathcal{X}$ is a pointed space, we use the notation $\Omega^1_s \mathcal{X}$ for the simplicial loop space of an $\aone$-fibrant resolution of $\mathcal{X}$, and $\Omega^1_{\gm{}}\mathcal{X}$ or $\Omega^1_{\pone}\mathcal{X}$ for the correspondingly defined $\gm{}$ or $\pone$-loop space of $\mathcal{X}$.

If $\mathcal{X}$ and $\mathcal{Y}$ are two spaces, we set $[\mathcal{X},\mathcal{Y}]_{\aone} := \hom_{\ho{k}}(\mathcal{X},\mathcal{Y})$ and refer to this set as the set of {\em free $\aone$-homotopy classes of maps from $\mathcal{X}$ to $\mathcal{Y}$}.  If $(\mathcal{X},x)$ and $(\mathcal{Y},y)$ are pointed spaces, then pointed $\aone$-homotopy classes of maps are denoted similarly with the base-point explicitly specified.  We define $\bpi_0^{\aone}({\mathcal X})$ as the Nisnevich sheaf associated with the presheaf $U \mapsto [U,\mathcal{X}]_{\aone}$.  If $(\mathcal{X},x)$ is pointed, then $\bpi_i^{\aone}({\mathcal X},x)$ is defined to be the Nisnevich sheaf associated with the presheaf $U \mapsto [S^i_s \sma U_+,(\mathcal X,x)]_{\aone}$ (and $U_+$ is $U \coprod \Spec k$).

A presheaf $\F$ on $\Sm_k$ is called $\aone$-invariant if the the map $\F(U) \to \F(U \times \aone)$ induced by the projection $U \times \aone \to U$ is a bijection.  A Nisnevich sheaf of groups $\mathbf{G}$ is called strongly $\aone$-invariant if $U \mapsto H^i_{\Nis}(U,\mathbf{G})$ is $\aone$-invariant for $i = 0,1$.  A Nisnevich sheaf of abelian groups $\mathbf{A}$ is called strictly $\aone$-invariant if $U \mapsto H^i_{\Nis}(U,\mathbf{A})$ is $\aone$-invariant for every integer $i \geq 0$. One of the main results of \cite{MField} that we use repeatedly below is \cite[Corollary 6.2]{MField}: if $k$ is a perfect field, then the sheaves $\bpi_i^{\aone}({\mathcal X},x)$ are strongly $\aone$-invariant for $i \geq 1$ and strictly $\aone$-invariant whenever they are abelian (e.g., for $i \geq 2$).\footnote{This result relies on \cite[Lemma 1.15]{MField}, the published version of which assumes $k$ is furthermore infinite.  Thus, a cautious reader may want to furthermore assume $k$ is infinite and perfect.}  Using the base-change results of Hoyois \cite[Lemma A.2 and A.4]{Hoyois}, if $\mathcal{X}$ arises as a pullback of a space defined over a perfect subfield, the hypothesis that $k$ is perfect may be dropped.  This hypothesis of being defined over a perfect subfield will be true in all situations of interest in this paper.

We will routinely use the fact that strongly (or strictly) $\aone$-invariant sheaves of groups are {\em unramified} in the sense of \cite[Definition 2.1]{MField}; this follows from \cite[Remark 6.10]{MField}.  Using \cite[Theorem 2.12]{MField}), a morphism $f: \mathbf{G} \to \mathbf{G}'$ of strongly $\aone$-invariant sheaves is an isomorphism (or mono- or epi-morphism) if and only if the induced map on sections over finitely generated extensions of the base field has the same property.   The full subcategory of the category of Nisnevich sheaves of abelian groups on $\Sm_k$ spanned by strictly $\aone$-invariant sheaves is an abelian subcategory by \cite[Lemma 6.2.13]{Morel05b}.

If $\mathbf{F}$ is a presheaf of abelian groups on $Sm_k$ and $X$ is a smooth scheme, we denote by $\widetilde {\mathbf{F}}(X)$ the cokernel of the map $\widetilde {\mathbf{F}}(\Spec k)\to \widetilde {\mathbf{F}}(X)$ induced by the structural morphism $X\to \Spec k$.  We call $\widetilde {\mathbf{F}}$ the reduced presheaf.

\section{Grothendieck--Witt groups}
\label{s:gwgroups}
In this section, we review some notation and results from the theory of higher Grothendieck--Witt groups which is a modern version of hermitian K-theory.  General references for this section are \cite{SchlichtingExact} and \cite{SchlichtingGWofschemes}.  We also collect some results, now available in \cite{SchlichtingTripathi} that give explicit spaces representing higher Grothendieck--Witt groups in the unstable (pointed) $\aone$-homotopy category.  To prove the main results stated in the introduction, we will need explicit descriptions of Grothendieck--Witt groups of low degree and we review the descriptions of some of these groups in terms of ``formations."  We will also require some ``explicit" descriptions of various homomorphisms between these groups, especially in low degrees.

\subsection{Recollections on the modern theory}
\label{ss:modern}
Let $k$ be a field having characteristic different from $2$, let $X$ be a smooth $k$-scheme and let $\L$ be a line bundle on $X$. Let $\mathcal C^b(X)$ be the exact category of bounded complexes of locally free coherent $\mathcal{O}_X$-modules and let $qis$ be the class of quasi-isomorphisms. The functor $\mathrm{Hom}_{\mathcal{O}_X}(\_,\L)$ on the category of  locally free coherent $\mathcal{O}_X$-modules induces a functor $\sharp_{\L}$ on $\mathcal C^b(X)$ , and the natural isomorphism $1\to \mathrm{Hom}_{\mathcal{O}_X}(\mathrm{Hom}_{\mathcal{O}_X}(\_,\L),\L)$ induced by the evaluation map yields an isomorphism $\varpi_{\L}:1\to \sharp_{\L}\sharp_{\L}$. The left translation $T^n:\mathcal C^b(X)\to \mathcal C^b(X)$ yields new dualities $\sharp_{\L}^n:=T^n\circ \sharp_{\L}$ and canonical isomorphisms $\varpi_{\L}^n:=(-1)^{n(n+1)/2}\varpi_{\L}$. For any $n\in \N$, the quadruple $(\mathcal C^b(X),qis,\sharp^n_{\L},\varpi_{\L}^n)$ is an \emph{exact category with weak-equivalences and (strong) duality} in the sense of \cite[\S 2.3]{SchlichtingGWofschemes}. Schlichting associates with such a category a Grothendieck--Witt space $\mathcal{GW}(\mathcal{C}^b(X),qis,\sharp^j_{\L},\varpi^j_{\L})$ \cite[\S 2.11]{SchlichtingGWofschemes}.  The higher Grothendieck--Witt groups of a smooth scheme are defined in terms of homotopy groups of this space.

\begin{defn}
For any integer $i\geq 0$, set $GW_i^n(X,\L) := \pi_i(\mathcal{GW}(\mathcal{C}^b(X),qis,\sharp^j_{\L},\varpi^j_{\L}))$; if $\L=\mathcal{O}_X$, then we set $GW_i^n(X) := GW_i^n(X,\mathcal{O}_X)$.
\end{defn}

\begin{rem}
The groups $GW_i^n(X,\L)$ are $4$-periodic in $n$ \cite[Remark 12]{SchlichtingGWofschemes} and the group $GW_0^n(X,\L)$ coincides with the Grothendieck--Witt group \`a la Balmer-Walter \cite[\S 2]{Walter03} of the triangulated category $D^b(X)$ of bounded complexes of coherent locally free $\mathcal{O}_X$-modules endowed with the duality $\sharp_{\L}^n$ and the natural isomorphism $\varpi_{\L}^n$ \cite[Lemma  8.2]{SchlichtingHKT}.  A helpful mnemonic, e.g., for the reader unfamiliar with Grothendieck--Witt groups, is to view the index $i$ in $GW^n_i(X)$ as playing the same role as the index $i$ in $K_i(X)$.
\end{rem}

Any map of smooth schemes $f:Y\to X$ induces a pullback morphism $f^*: GW_i^n(X,\L) \to GW_i^n(Y,f^* \L)$.  In the special case where $j: U \to X$ is an open immersion, the induced map $j^*: GW_0^n(X) \to GW_0^n(U)$ is not surjective (in contrast, e.g., to the situation in algebraic K-theory).  To measure this failure of surjectivity and also to obtain a localization long exact sequence, Schlichting defines negative Grothendieck--Witt groups.  To this end, Schlichting associates a spectrum ${\mathbb G}W^n(X,\L)$ with the category $(\mathcal C^b(X),qis,\sharp^n_{\L},\varpi_{\L}^n)$ and defines negative Grothendieck--Witt groups by means of the formula $GW_{-i}^n(X,\L) := \pi_{-i}({\mathbb G}W^n(X,\L))$ \cite[\S 10]{SchlichtingGWofschemes}; the resulting spectrum, of course, recovers the definition given above in positive degrees.  Write $W^i(X,\L)$ for the Balmer--Witt group \cite[\S 1.4]{Balmer05b} of a scheme $X$.  For $i \geq 0$ the formula $GW_{-i}^n(X,\L)=W^{i+n}(X,\L)$ identifies the negative Grothendieck--Witt groups as defined above as Balmer--Witt groups \cite[Propositions 6.3 and 9.3]{SchlichtingHKT}.

When $X=\Spec (R)$, the groups defined above coincide with hermitian $K$-theory groups as defined by M. Karoubi \cite{Karoubi73,Karoubi80}.  More precisely, Karoubi considers the spaces $GW(R) \times BO(R)^+$ and ${}_{-1}GW(R) \times BSp(R)^+$ and defines $KO_i(R) = \pi_i(GW(R) \times BO(R)^+)$ and $KSp_i(R) = \pi_i({}_{-1}GW(R) \times BSp(R)^+)$; here $GW(R)$ is the usual Grothendieck--Witt group of isomorphism classes of non-degenerate symmetric bilinear forms over $R$ and ${}_{-1}GW(R)$ is the Grothendieck--Witt group of non-degenerate symplectic spaces over $R$ (note: our notation differs slightly from that of Karoubi, but we hope it is more suggestive). Schlichting shows in \cite[Corollary A.2]{SchlichtingHKT} that there are canonical isomorphisms of the form
\[
\begin{split}
GW_i^0(R) \cong KO_i(R) \\
GW_i^2(R) \cong KSp_i(R).
\end{split}
\]
Karoubi defines $U$-theory and $V$-theory groups in terms of homotopy fibers of natural ``hyperbolic" and ``forgetful" maps (see below) and the groups $GW^1_i(R)$ and $GW^3_i(R)$ coincide with the groups ${}_{-1}U_i(R)$ and $U_i(R)$ respectively \cite[Theorems 6.1-2]{SchlichtingHKT}.

Grothendieck--Witt groups are equipped with a multiplicative structure by \cite[\S 9.2]{SchlichtingHKT}. More precisely, the tensor product of complexes induces product maps
\[
\cup:GW_i^j(X,\L_1)\times GW_r^s(X,\L_2) \longrightarrow GW_{i+r}^{j+s}(X,\L_1\otimes \L_2)
\]
for any $i,j,r,s\in\Z$.  These product maps are graded commutative in the sense that $\alpha\cup \beta=(-1)^{ir}\langle -1\rangle^{js}\beta\cup\alpha$ and the class $\langle 1\rangle$ is a unit for this product.

There is a canonical identification $GW_{-1}^{-1}(k) \cong W(k)$ and the class of $\langle 1 \rangle$ defines an element $\eta \in GW_{-1}^{-1}(k)$.  Taking the product with $\eta$ defines a homomorphism $\eta: GW^j_i(X,\L) \to GW^{j-1}_{i-1}(X,\L)$ for any smooth scheme $X$ and line bundle $\L$ on $X$ (see the construction at the beginning of \cite[\S 6]{SchlichtingHKT}).  For any $i\in\N$, if $K_i(X)$ is the usual Quillen $K$-theory group of $X$, then for any $n \in \N$, there are hyperbolic morphisms
\[
H_{i,n}:K_i(X)\to GW_i^n(X,\L)
\]
and forgetful morphisms
\[
f_{i,n}:GW_i^n(X,\L)\to K_i(X).
\]
By \cite[Theorem 6.1]{SchlichtingHKT}, the hyperbolic and forgetful homomorphisms fit into long exact sequences of the form
\[
\xymatrix@C=1.71em{\ldots\ar[r] & K_i(X)\ar[r]^-{H_{i,j}} & GW_i^j(X,{\L})\ar[r]^-{\eta} & GW_{i-1}^{j-1}(X,{\L})\ar[r]^-{f_{i-1,j-1}} & K_{i-1}(X)\ar[r]^-{H_{i-1,j}} & GW_{i-1}^j(X,{\L})\ar[r] & \ldots};
\]
we will refer to these sequences either as \emph{Bott sequences} or \emph{Karoubi periodicity sequences}.  The following lemma, which records a computation about composition of hyperbolic and forgetful morphisms, will be used in Section \ref{section:Suslindegree}.

\begin{lem}
\label{lem:comparison}
For any $i,j\in\N$, any smooth $k$-scheme $X$ and any line bundle $\L$ on $X$, the composite
\[
GW_i^j(X,{\L})\stackrel{f_{i,j}}{\longrightarrow} K_i(X) \stackrel{H_{i,j}}{\longrightarrow} GW_i^j(X,{\L})
\]
coincides with multiplication by $\langle 1,-1\rangle\in GW_0^0(k)$.
\end{lem}

\begin{proof}
We first treat the case where $i=j=0$, $X=\Spec (k)$ and $\L=k$. In that case, the forgetful map $f_{0,0}:GW_0^0(k)\to \Z$ is simply the rank map, while the hyperbolic map $H_{0,0}:\Z\to GW_0^0(k)$ sends $n\mapsto n\cdot\langle 1,-1\rangle$. In particular, the conclusion of the lemma is valid for $\langle 1\rangle\in GW_0^0(k)$.

To treat the general case, it suffices to establish that for a smooth $k$-scheme $X$ any line bundle $\mathcal N$ over $X$, any $\alpha\in GW_i^j(X,\L)$ and any $\beta\in GW_r^s(X,\mathcal N)$, the equality $H_{i+r,j+s}f_{i+r,j+s}(\alpha\cdot \beta)=H_{i,j}f_{i,j}(\alpha)\cdot \beta$ holds. Indeed, we then have
\[
H_{i,j}f_{i,j}(\alpha)=H_{i,j}f_{i,j}(\langle 1\rangle\cdot\alpha)=H_{0,0}f_{0,0}(\langle 1\rangle)\cdot \alpha=\langle 1,-1\rangle\cdot \alpha.
\]
Recall from \cite[\S 2.15]{SchlichtingGWofschemes} that the hyperbolic category $\mathcal H\mathcal C^b(X)$ associated with the exact category with weak-equivalences and duality $(\mathcal C^b(X),qis,\sharp^n_{\L},\varpi_{\L}^n)$ has objects of the form $(A,B)$ where $A$ and $B$ are objects of $\mathcal C^b(X)$ and morphisms $(A,B)\to (A^\prime,B^\prime)$ are given by pairs $(f,g)$ of morphisms of $\mathcal C^b(X)$ with $f:A\to A^\prime$ and $g:B^\prime\to B$. The category $\mathcal H\mathcal C^b(X)$ is endowed with a duality $\star$ defined by $(A,B)^\star=(B,A)$ and the identity as canonical isomorphism.

The Grothendieck--Witt space $\mathcal{GW}(\mathcal H\mathcal{C}^b(X),qis,*,Id)$ is naturally homotopy equivalent to the $K$-theory space $\mathcal K(\mathcal{C}^b(X),qis)$ \cite[Proposition 2.17]{SchlichtingGWofschemes}. In this context, the forgetful functor $f_j:(\mathcal C^b(X),qis,\sharp^n_{\L},\varpi_{\L}^n)\to (\mathcal H\mathcal{C}^b(X),qis,*,Id)$ is defined by $f_j(A)=(A,A^{\sharp^n_{\L}})$ and $f_j(\alpha)=(\alpha,\alpha^{\sharp^n_{\L}})$ for any morphism $\alpha:A\to B$, while the hyperbolic functor $H_j:(\mathcal H\mathcal{C}^b(X),qis,*,Id)\to (\mathcal C^b(X),qis,\sharp^n_{\L},\varpi_{\L}^n)$ is defined on objects by $H_j(A,B)=A\oplus B^{\sharp^n_{\L}}$ and on morphisms by $H_j(\alpha,\beta)=\alpha\oplus \beta^{\sharp^n_{\L}}$. The duality transformations $\eta_j:*\circ f_j\to f_j\circ \sharp_{\L}^n$ and $\mu_j:\sharp_{\L}^m\circ H_j\to H_j\circ *$ are respectively given by the matrices $\begin{pmatrix} 1 & 0 \\ 0 & \varpi_{\L}^n\end{pmatrix}$ and $\begin{pmatrix} 0 & (\varpi_{\L}^n)^{-1} \\ 1 & 0\end{pmatrix}$.

As a consequence, $H_jf_j(A)=A\oplus A^{\sharp^n_{\L}\sharp^n_{\L}}$ and $H_jf_j(\alpha)=\alpha\oplus \alpha^{\sharp^n_{\L}\sharp^n_{\L}}$ with natural transformation $\sharp_{\L}^n\circ (H_jf_j)\to (H_jf_j)\circ \sharp_{\L}^n$ given by the matrix $\begin{pmatrix} 0 & (\varpi_{\L}^n)^{-1} \\ \varpi_{\L}^n & 0\end{pmatrix}$. Consider now the functor $G_j:(\mathcal C^b(X),qis,\sharp^n_{\L},\varpi_{\L}^n)\to (\mathcal C^b(X),qis,\sharp^n_{\L},\varpi_{\L}^n)$ defined on objects by $G_j(A)=A\oplus A$ and on morphisms by $G_j(\alpha)=\alpha\oplus\alpha$. This functor can be made duality preserving by considering the natural transformation $\sharp_{\L}^n\circ G_j\to G_j\circ \sharp_{\L}^n$ induced by the matrix $\begin{pmatrix} 0 & 1 \\ 1 & 0 \end{pmatrix}$.  There is a natural transformation
\[
\begin{pmatrix} 1 & 0 \\ 0 & \varpi_{\L}^n\end{pmatrix}:G_j \longrightarrow H_jf_j
\]
which is easily seen to be a weak-equivalence of form functors in the sense of \cite[\S 2.1]{SchlichtingGWofschemes}. It follows from \cite[Lemma 2]{SchlichtingGWofschemes} that the maps on Grothendieck--Witt spaces induced by these functors are homotopic. Now $G_j$ commutes with tensor products of complexes in an obvious sense, and it follows therefore that the equality $H_{i+r,j+s}f_{i+r,j+s}(\alpha\cdot \beta)=H_{i,j}f_{i,j}(\alpha)\cdot \beta$ holds for any $\alpha\in GW_i^j(X,\L)$ and any $\beta\in GW_r^s(X,\mathcal N)$.
\end{proof}

\subsection{Geometric representability}
As higher Grothendieck--Witt groups satisfy Nisnevich descent and are $\aone$-homotopy invariant, it turns out that they are representable in the $\aone$-homotopy category (see, e.g., \cite[Theorem 3.1]{Hornbostel}). In \cite{SchlichtingTripathi}, it is proven that the spaces representing higher Grothendieck--Witt have the $\aone$-homotopy type of explicit ind-algebraic varieties, analogous to the situation in classical algebraic topology.

First, we recall the geometric representability results for $GW^{2j}_i$.  We refer to \cite{SchlichtingTripathi} for the construction of the ``infinite orthogonal Grassmannian" $OGr$ and \cite{PaninWalterBO} for the construction of the ``infinite symplectic Grassmannian".  Using these constructions, for any smooth $k$-scheme $X$ there are functorial isomorphisms
\[
\begin{split}
[S^i_s\wedge X_+,\Z\times OGr]_{\aone} &\simeq GW_i^0(X), \text{ and } \\
[S^i_s\wedge X_+,\Z\times HGr]_{\aone} &\simeq GW_i^2(X);
\end{split}
\]
for the first statement we refer the reader to \cite[Proposition 8.1, Theorem 8.2]{SchlichtingTripathi}, and for the second statement, the reader may consult \cite[Theorem 8.2]{PaninWalterBO}.  By adjunction, one deduces that $\Omega^i_s(\Z \times OGr)$ and $\Omega^i_s(\Z \times HGr)$ represent the functors $X \mapsto GW^0_i(X)$ and $X \mapsto GW^2_i(X)$.

To state the representability results for $GW^{2j+1}_i(X)$ requires introducing a bit more notation.  If $G$ is a linear algebraic group over $k$, and $H$ is a closed subgroup, then a quotient of $G$ by $H$ exists in the category of schemes \cite[Expos{\'e} V Th{\'e}or{\`e}me 10.1.2]{SGA31}.  If $G$ and $H$ are smooth, then it is known that this scheme quotient coincides with the \'etale sheafification of the presheaf $X \mapsto G(X)/H(X)$.  For this reason we write, following Schlichting-Tripathi, $G/H_{\et}$ for the scheme-theoretic quotient and, abusing terminology, for the restriction of $G/H_{\et}$ to the Nisnevich site as well.  On the other hand, we use the notation $G/H$ for the {\em Nisnevich} sheaf quotient, i.e., the Nisnevich sheaf associated with the presheaf $X \mapsto G(X)/H(X)$.  There is a canonical map $G/H \to G/H_{\et}$, which need not be an isomorphism in general.  If $H$ is special in the sense of Serre, i.e., if all $H$-torsors are Zariski locally trivial, then the canonical map $G/H \to G/H_{\et}$ can be seen to be an isomorphism using Nisnevich local sections; this happens, e.g., if $G = GL_n, SL_n$ or $Sp_{2n}$.

For any $n\in\N$, consider the closed embeddings $GL_{n}\to O_{2n}$ and $GL_{n}\to Sp_{2n}$ defined by
\[
M \longmapsto \begin{pmatrix} M & 0 \\ 0 & (M^{-1})^t \end{pmatrix}.
\]
These inclusions are compatible with standard stabilization embeddings $GL_{n} \hookrightarrow GL_{n+1}$, $O_{2n} \hookrightarrow O_{2n+2}$, and $Sp_{2n} \hookrightarrow Sp_{2n+2}$ given by viewing the subgroups as block submatrices.  Taking colimits with respect to the induced maps $(O_{2n}/GL_{n}) \to (O_{2n+2}/GL_{n+1})$ or $Sp_{2n}/GL_n \to Sp_{2n+2}/GL_{n+1}$, we obtain spaces $O/GL$ and $Sp/GL$ (here, as mentioned in the above paragraph, the subscript $\et$ would be redundant).  For any smooth scheme $X$, we have natural isomorphisms
\[
\begin{split}
[S^i_s\wedge X_+,Sp/GL]\simeq GW_i^1(X) \\
[S^i_s\wedge X_+,O/GL]\simeq GW_i^3(X).
\end{split}
\]
by \cite[Theorem 8.4]{SchlichtingTripathi}.

Analogously, there are standard inclusion homomorphisms $O_{2n} \hookrightarrow GL_{2n}$ and $Sp_{2n} \hookrightarrow GL_{2n}$ that are similarly compatible with stabilization embeddings.  We obtain maps $(GL_{2n}/O_{2n})_{\et} \hookrightarrow (GL_{2n+2}/O_{2n+2})_{\et}$ and set $GL/O_{\et} := \colim_n (GL_{2n}/O_{2n})_{\et}$ with respect to these embeddings.  Similarly, we obtain maps $GL_{2n}/Sp_{2n} \to GL_{2n+2}/Sp_{2n+2}$ and we set $GL/Sp := \colim_n GL_{2n}/Sp_{2n}$.

\begin{rem}
Strictly speaking, Schlichting and Tripathi define $(GL/O)_{\et}$ as the ``\'etale quotient" of the stable group $GL$ by the stable group $O$ and provide a similar definition for $GL/Sp$.  That our definition of $GL/Sp$ coincides with the Schlichting-Tripathi definition follows from the fact that all quotients are taken in the Nisnevich topology so we can commute the (homotopy) colimits.  Likewise, our definition of $(GL/O)_{\et}$ coincides with their definition since, in this case, the quotients are formed in the category of \'etale sheaves.
\end{rem}

The next result provides the analog of the geometric form of Bott periodicity in unstable $\aone$-homotopy theory.

\begin{thm}[{\cite[Theorems 8.2 and 8.4]{SchlichtingTripathi}}]
\label{thm:loopspacemodels}
There are canonical $\aone$-weak equivalences
\[
\Omega^n_{\pone}(\Z \times OGr) \cong \begin{cases} \Z \times OGr & \text{ if } n\equiv 0 \mod 4 \\
O/GL & \text{ if } n \equiv 1 \mod 4 \\
\Z \times HGr & \text{ if } n \equiv 2 \mod 4 \\
Sp/GL & \text{ if } n \equiv 3 \mod 4,
\end{cases}
\]
and
\[
\Omega^1_s\Omega^n_{\pone} (\Z \times OGr) \cong \begin{cases} O & \text{ if } n\equiv 0 \mod 4 \\
GL/Sp & \text{ if } n \equiv 1 \mod 4 \\
Sp & \text{ if } n \equiv 2 \mod 4 \\
(GL/O)_{et} & \text{ if } n \equiv 3 \mod 4.
\end{cases}
\]
\end{thm}

The geometric models for the various loop spaces of $\Z \times OGr$ provide explicit de-loopings of the space $(\Z \times OGr)$, and we therefore make the following definition.

\begin{defn}
For any $n\in\N$, we set
\[
\Omega_{\pone}^{-n}(\Z \times OGr)=\begin{cases} \Z \times OGr & \text{ if } n \equiv 0 \mod 4 \\
Sp/GL & \text{ if } n \equiv 1 \mod 4 \\
\Z \times HGr & \text{ if } n \equiv 2 \mod 4 \\
O/GL & \text{ if } n \equiv 3 \mod 4
\end{cases}
\]
and
\[
\Omega_{\pone}^{-n}O= \begin{cases} O & \text{ if } n\equiv 0 \mod 4 \\
(GL/O)_{et} & \text{ if } n \equiv 1 \mod 4 \\
Sp & \text{ if } n \equiv 2 \mod 4 \\
GL/Sp & \text{ if } n \equiv 3 \mod 4.
\end{cases}
\]
\end{defn}

\subsection{Explicit groups}
\label{s:classical}
If $R$ is a smooth $k$-algebra, then there are explicit descriptions of the groups $GW_i^n(R)$ for $i = 0,1$ that were studied ``classically"; in this section we review these constructions.  In more detail, if $P$ is a projective module, and we write $P^{\vee} = \hom_R(P,R)$ for its $R$-module dual, then there is a canonical evaluation isomorphism $\varepsilon_P: P \longrightarrow P^{\vee\vee}$.  If we write $\mathcal{P}_R$ for the category of projective $R$-modules, then the triple $(\mathcal{P}_R,(\cdot)^{\vee},\varepsilon)$ is an exact category with duality in the sense of, e.g., \cite[Definition 2.1]{SchlichtingExact}.  In any exact category with duality $(\mathcal{E},\ast,\tau)$ , one can speak of symmetric forms, isotropic subspaces, Lagrangian subspaces and metabolic spaces \cite[Definition 2.4-5]{SchlichtingExact}.  In this context, one can define Grothendieck--Witt groups $GW_0(\mathcal{E},\ast,\tau)$.

\subsubsection{Degree $0$}
One defines
\[
\begin{split}
GW^0_0(R) &:= GW_0(\mathcal{P}_R,(\cdot)^{\vee},\varepsilon), \text{ and } \\
GW^2_0(R) &:= GW_0(\mathcal{P}_R,(\cdot)^{\vee},-\varepsilon)
\end{split}
\]
(see, e.g., \cite[\S 2.2]{SchlichtingExact} for more details).  The group $GW^0_0(R)$ can be thought of as the Grothendieck group of isometry classes of symmetric bilinear forms, while $GW^2_0(R)$ can be thought of as the Grothendieck group of isometry classes of anti-symmetric bilinear forms.

One can also define ``formations" in any exact category with duality $(\mathcal{E},\ast,\tau)$: a {\em formation} is a quadruple $(X,\varphi,L_1,L_2)$, where $(X,\varphi)$ is a metabolic space and $L_1$ and $L_2$ are two Lagrangian subspaces of $(X,\varphi)$.  An isometry between formations is an isometry of symmetric spaces that preserves the associated Lagrangians, and one can define an orthogonal direct sum of formations.  One can define the {\em Grothendieck--Witt group of formations} $GW_{form}(\mathcal{E},\ast,\psi)$ (\cite[\S 4.3]{SchlichtingExact} for instance) as the quotient of the free abelian group generated by isometry classes $[X,\varphi,L_1,L_2]$ of formations by the relations
\begin{enumerate}[noitemsep,topsep=1pt]
\item $[X,\varphi,L_1,L_2]+[Y,\psi,N_1,N_2]=[X\oplus Y, \varphi\perp \psi,L_1\oplus N_1, L_2\oplus N_2]$,
\item $[X,\varphi,L_1,L_2]+[X,\varphi,L_2,L_3]=[X,\varphi,L_1,L_3]$,
\item If $L\subset X$ is a sub-Lagrangian which is an admissible sub-object of both $L_1$ and $L_2$, then $[X,\varphi,L_1,L_2]=[L^\perp/L,\overline\varphi,L_1/L,L_2/L]$, where $\overline \varphi$ is the symmetric isomorphism induced on $L^\perp/L$ by $\varphi$.
\end{enumerate}
We then set
\[
\begin{split}
GW^1_0(R) &:= GW_{form}((\mathcal{P}_R,(\cdot)^{\vee},-\varepsilon)), \text{ and }\\
GW^3_0(R) &:= GW_{form}((\mathcal{P}_R,(\cdot)^{\vee},\varepsilon)).
\end{split}
\]
That the groups defined above coincide with the more abstract groups described in Subsection \ref{ss:modern} is a consequence of \cite[Theorems 7.1 and 8.1]{Walter03}.

\subsubsection{Degree $1$}
\label{ss:degree1}
As mentioned in Subsection \ref{ss:modern}, one knows that $GW_1^0(R)=KO_1(R)=O(R)/[O(R),O(R)]$ and $GW_1^2(R)=KSp_1(R)=Sp(R)/[Sp(R),Sp(R)]$.  As usual, write $E_n(R) \subset GL_n(R)$ for the subgroup of elementary matrices. We set $EO_{2n}(R) := O_{2n}(R) \cap E_{2n}(R)$ and $ESp_{2n}(R) := Sp_{2n}(R) \cap E_{2n}(R)$.  There are induced stabilization homomorphisms $EO_{2n}(R) \hookrightarrow EO_{2n+2}(R)$ and $ESp_{2n}(R) \hookrightarrow ESp_{2n+2}(R)$ and we can use these homomorphisms to define stable groups $EO(R)$ and $ESp(R)$.  The group of elementary orthogonal or symplectic matrices is a subgroup of the corresponding commutator subgroup, and Vaserstein showed \cite{Vaserstein} that stably the two subgroups are equal, i.e.,
\[
\begin{split}
O(R)/EO(R) &\isomto GW^0_1(R), \text{ and }\\
Sp(R)/ESp(R) &\isomto GW^2_1(R);
\end{split}
\]
(see also \cite[Chapter II, Theorem 5.2]{BassUnitary}).

Let $\tau_{2n}$ be the $2n \times 2n$ matrix (over $\Z$) that is the $n$-fold block sum of $2 \times 2$ matrices of the form
\[
\tau_2 := \begin{pmatrix}0 & 1 \\ 0 & 0 \end{pmatrix}
\]
and let $\sigma_{2n}=\tau_{2n}+\tau_{2n}^t$.

Next, let $S_{2n}$ be the scheme of invertible symmetric $2n \times 2n$-matrices.  Define an action of $GL_{2n}$ on $S_{2n}$ by the formula $g \cdot X = g^t X g$.  The stabilizer of $\sigma_{2n}$ under this action is, by definition, $O_{2n}$ and thus the orbit through $\sigma_{2n}$ defines a morphism $GL_{2n}\to S_{2n}$ that factors through a morphism $(GL_{2n}/O_{2n})_{et}\to S_{2n}$.  We claim that this map is an isomorphism of schemes.  If $R$ is a local regular $k$-algebra, any bilinear symmetric form is isometric to a diagonal matrix by \cite[II.3 Corollary to Proposition 1]{Knebusch76}.  If $R$ is, furthermore, strictly henselian, and $D$ is a diagonal matrix over $R$, then square roots of the reductions of the diagonal entries of $D$ to the residue field of $R$ can be lifted to $R$ by Hensel's lemma to produce an isometry of $D$ with $\sigma_{2n}$. It follows that the map $(GL_{2n}/O_{2n})_{et}\to S_{2n}$ is an isomorphism of \'etale sheaves and thus of schemes.

There are morphisms of schemes $S_{2n} \to S_{2n+2}$ defined by sending $M \in S_{2n}(R)$ to the block sum $M \perp \sigma_2$.  Fixing the basepoint $\sigma_2 \perp \cdots \perp \sigma_2$, the stabilization map coincides with the map $(GL_{2n}/O_{2n})_{et}\to (GL_{2n+2}/O_{2n+2})_{et}$ induced by the inclusion $GL_{2n}\to GL_{2n+2}$ mapping a matrix to the block matrix $\mathrm{diag}(M,Id_2)$. If we set $S := \colim_n S_{2n}$, then we see that $S\cong (GL/O)_{et}$.

Say that two invertible symmetric matrices $M \in S_{2m}(R)$ and $N \in S_{2n}(R)$ are stably equivalent, and write $M \sim N$, if there is an integer $t$ such that the block sum $M \perp \sigma_{2n+2t}$ is conjugate by an element of $E_{2m+2n+2t}(R)$ to $N \perp \sigma_{2m+2t}$.  The block sum of matrices then yields a well-defined (commutative) monoid structure on $S(R)$ with unit given by the class of $\sigma_2$ in $S(R)$.  By, e.g., \cite[Lemme 4.5.1.9]{BargeLannes} one knows that $M \perp (-M)^{-1} \sim \sigma_2$, so $S(R)/\sim$ is actually an abelian group.  By \cite[Theorem 8.4]{SchlichtingTripathi}, we know that $GW_1^1(R)\cong S(R)/\sim$.

In a similar manner, write $A_{2n}(R)$ for the scheme of invertible $2n\times 2n$ anti-symmetric matrices and set $\psi_{2n}:=\tau_{2n}-\tau_{2n}^t$. We define a map $GL_{2n}\to A_{2n}$ by $M\mapsto M^t\psi_{2n}M$, which induces a map $GL_{2n}/Sp_{2n}\to A_{2n}$. If $R$ is a local $k$-algebra, then every anti-symmetric matrix is isometric to $\psi_{2n}$ and we therefore obtain an isomorphism $GL_{2n}/Sp_{2n}\to A_{2n}$. Arguing as before, we get an isomorphism $A\cong GL/Sp$ where $A:=\colim_n A_{2n}$ (the transition maps are defined by adding $\psi_2$). As before, we can say that two invertible anti-symmetric matrices $M \in A_{2m}(R)$ and $N \in A_{2n}(R)$ are stably equivalent, and again we write $M \sim N$, if there is an integer $t$ such that the block sum $M \perp \psi_{2n+2t}$ is conjugate by an element of $E_{2m+2n+2t}(R)$ to $N \perp \psi_{2m+2t}$.  The block sum of matrices then yields a well-defined (commutative) monoid structure on $A(R)/\sim$ with unit the class of $\psi_2$ in $A(R)$.  By, e.g., \cite[\S 3]{VasersteinSuslin}, $M \perp \sigma_{2m}M^{-1}\sigma_{2m} \sim \psi_2$ and it follows that $A(R)/\sim$ is actually an abelian group. Again, \cite[Theorem 8.4]{SchlichtingTripathi} yields an isomorphism $GW_1^3(R)\cong A(R)/\sim$.

\subsection{Periodicity homomorphisms}
\label{subsec:periodicity}
Now that we have concrete descriptions for low degree Grothendieck--Witt groups, we give corresponding concrete descriptions of the periodicity homomorphisms
\[
\eta:GW_1^n(R) \longrightarrow GW_0^{n-1}(R).
\]
To this end, we unwind the constructions in the proof of \cite[Theorem 6.1]{SchlichtingHKT}.  Indeed, Schlichting constructs the Bott sequences using a commutative diagram: traversing the diagram in one way corresponds to the definition given above involving multiplication by $\eta$, while traversing in the other way can be given a more explicit description involving the cone of a certain explicit map of spectra.  Using this diagram, Schlichting constructs a long exact sequence, and then explicitly identifies the connecting homomorphisms. Leaving the second step aside, since it is not necessary for our purposes, the exact sequences of the first step can be obtained in low degrees using the procedure of \cite[VII \S 5]{BassKTheory} (see \cite[VII Theorem 5.3]{BassKTheory} for an explicit exact sequence) or \cite[II, 2.13]{Karoubi78}.  Unwinding the details, we now give explicit formulas for $\eta$ in each case.

In case $n=1,3$ (up to modernizing the notation a bit) there are exact sequences
\[
\xymatrix{GW_1^{n-1}(R)\ar[r]^-{f_{1,n-1}} & K_1(R)\ar[r]^-{H_{1,n}} & _{\epsilon}V(R)\ar[r]^-\eta & GW_0^{n-1}(R)\ar[r]^-{f_{0,n-1}} & K_0(R)}
\]
where $ _\epsilon V(R)$ are Karoubi $V$ groups as defined in \cite[\S II]{Karoubi73} with $\epsilon=(-1)^{n(n-1)/2}$. To obtain the explicit descriptions below, it suffices then to identify the groups $_{\epsilon}V(R)$ with $GW_1^{n}(R)$ in the spirit of \cite[Proposition-D\'efinition 4.5.2.2 (d)]{BargeLannes} for $n=1$ and \cite[\S 4.3]{Fasel10b} for $n=3$.

More concretely, if $M\in S_{2n}(R)$, then we can consider the class $[M]-[\sigma_{2n}]$ in $GW^0_0(R)$. Stabilizing, yields a well-defined homomorphism $\eta: GW_1^1(R)\to GW^0_0(R)$ which coincides with $\eta$.  Similarly, given $M \in A_{2n}(R)$, we may consider the class $[M] - [\psi_{2n}]$ in $GW^2_0(R)$, and stabilization yields the periodicity homomorphism $\eta: GW_1^3(R) \to GW^2_0(R)$.

The cases $n=0,2$ are slightly more involved. Then, \cite[\S II]{Karoubi73} yields exact sequences
\[
\xymatrix{K_1(R)\ar[r]^-{H_{1,n}} & GW_1^n(R)\ar[r]^-{\eta} & _{\epsilon}U(R)\ar[r]^-{f_0} & K_0(R)\ar[r]^-{H_{0,n}} & GW_0^n(R)}
\]
where $ _{\epsilon}U(R)$ are Karoubi $U$-groups as defined in \cite[\S II]{Karoubi73} with $\epsilon=(-1)^{n(n-1)/2}$. The description of $\eta$ follows from the identification of $ _\epsilon U(R)$ with formations as defined in Section \ref{s:classical} in the spirit of \cite[\S 4.3]{SchlichtingExact}.

Suppose now that $M\in O_{2n}(R)$. The form $\sigma_{2n}:R^{2n}\to (R^{2n})^{\vee}$ has a Lagrangian
\[
i_n:R^n \longrightarrow R^{2n}
\]
defined by $(a_1,\ldots,a_n)\mapsto (a_1,0,a_2,0,\ldots,a_n,0)$. It is clear that $Mi_n$ is again a Lagrangian of $\sigma_{2n}$ and we can consider the class $[R^{2n},\sigma_{2n},i_n(R^n),Mi_n(R^n)]$ in $GW^3_0(R)$.

Since we have $[R^2,\sigma_2,i_1(R),i_1(R)]=0$, it is clear that the map $O_{2n}(R)\to GW^3_0(R)$ defined by $M\mapsto [R^{2n},\sigma_{2n},i_n(R^n),Mi_n(R^n)]$ stabilizes to a map $O(R)\to GW^3_0(R)$. If $N\in O_{2n}(R)$, we obtain an isometry
\[
(R^{2n},\sigma_{2n},i_n(R),Mi_n(R))\simeq (R^{2n},\sigma_{2n},Ni_n(R^n),NMi_n(R^n))
\]
showing that $[R^{2n},\sigma_{2n},i_n(R),Mi_n(R)]=[R^{2n},\sigma_{2n},Ni_n(R^n),NMi_n(R^n)]$ in $GW^3_0(R)$.  Then, the following formulas also hold.
\[
\begin{split}
[R^{2n},\sigma_{2n},Ni_n(R^n),NMi_n(R^n)]&=[R^{2n},\sigma_{2n},Ni_n(R^n),i_n(R^n)]+[R^{2n},\sigma_{2n},i_n(R^n),NMi_n(R^n)] \text{ and } \\
[R^{2n},\sigma_{2n},Ni_n(R^n),i_n(R^n)]&=-[R^{2n},\sigma_{2n},i_n(R^n),Ni_n(R^n)],
\end{split}
\]
As a consequence, observe that we obtain a morphism
\[
\begin{split}
GW^0_1(R) &\longrightarrow GW^3_0(R) \\
M &\longmapsto [R^{2n},\sigma_{2n},i_n(R^n),Mi_n(R^n)],
\end{split}
\]
which is precisely $\eta$.

After replacing orthogonal groups by symplectic groups and symmetric forms by anti-symmetric forms, we may repeat the arguments just given.  In that case, for $M \in Sp_{2n}(R)$ we obtain homomorphisms
\[
\begin{split}
  GW^2_1(R) &\longrightarrow GW^1_0(R) \\
   M   &\longmapsto [R^{2n},\psi_{2n},i_n(R^n),Mi_n(R^n)]
\end{split}
\]
yielding $\eta$ in this degree.

The constructions above depended on a particular choice of Lagrangian subspace $i_n$; we now observe that changing the Lagrangian subspace yields alternative maps that we may compare to $\eta$.  In that direction, suppose that $j_n:R^n\to R^{2n}$ is some fixed Lagrangian of $\sigma_{2n}$.  Define maps $\eta': GW_1^0(R) \to GW^3_0(R)$ on $M \in O_{2n}(R)$ by means of the formula
\[
\eta^{\prime}(M) := [R^{2n},\sigma_{2n},j_n(R^n),Mj_n(R^n)].
\]
Likewise, abusing notation slightly, suppose $j_n: R^n \to R^{2n}$ is a fixed Lagrangian for $\psi_{2n}$.  Then, for $M \in Sp_{2n}(R)$, define maps $\eta': GW^2_1(R) \to GW^1_0(R)$ by means of the formula:
\[
\eta^{\prime}(M) := [R^{2n},\psi_{2n},j_n(R^n),Mj_n(R^n)].
\]
The next result analyzes the relationship between the maps $\eta$ and $\eta'$: in essence it shows that choice of Lagrangian does not affect the map.

\begin{lem}
\label{lem:otherLagrangian}
The map $\eta^{\prime}$ coincides with $\eta$.
\end{lem}

\begin{proof}
We treat the symplectic case first.  In view of \cite[Proposition 2.1.5(c)]{BargeLannes}, since $j_n$ defines a free $R$-submodule of $R^{2n}$ by definition, there exists a symplectic matrix $G\in Sp_{2n}(R)$ such that $G i_n=j_n$. Therefore, the following equalities hold:
\[
\eta^\prime(M)=[R^{2n},\psi_{2n},G^{-1}j_n(R^n),G^{-1}Mj_n(R^n)]=[R^{2n},\psi_{2n},i_n(R^n),G^{-1}MGi_n(R^n)].
\]
Now $[R^{2n},\psi_{2n},i_n(R^n),G^{-1}MGi_n(R^n)]=\eta(G^{-1})+\eta(MG)=\eta(MG)+\eta(G^{-1})=\eta(M)$.

The orthogonal case is established in an analogous fashion.  One begins by establishing a suitable orthogonal analogue of \cite[Proposition 2.1.5]{BargeLannes}.  Precisely, if $X$ is a Lagrangian of $\sigma_{2n}$, then the following conditions are equivalent (i) $X$ is a free $R$-module, (ii) there is an $R$-module isomorphism $X \cong i_{n}(R^n)$, and (iii) there exists $G \in O_{2n}(R)$ such that $X = Gi_{n}(R^n)$.  To establish this result, one simply repeats the proof of \cite[Proposition 2.1.5]{BargeLannes} replacing each occurrence of {\em symplectic} by {\em orthogonal}.
\end{proof}

\subsection{Gysin homomorphisms}
\label{subsec:Gysin}
We now recall some facts about transfer maps for Grothendieck--Witt groups.  As usual, let $k$ be a field having characteristic different from $2$. Suppose $f:R\to S$ is a $k$-algebra homomorphism making $S$ into a finitely generated $R$-module. Suppose furthermore that $R$ and $S$ are smooth, integral $k$-algebras and set $d=\dim R-\dim S$.  In this situation, the $S$-module $\ext^d_R(S,R)$ is invertible (this follows from \cite[Corollary 6.3]{Gille03}) and one may define a transfer map (\cite[Theorem 6.4]{Gille03} or \cite[Theorem 4.4]{Calmes08})
\[
f_*:GW^n_0(S,\ext^d_R(S,R)) \longrightarrow GW^{n+d}_0(R).
\]
This morphism can be made more explicit in the situation where $S = R/a$ where $a \in R$ is such that $S$ is smooth and integral with $\dim (S)=\dim (R)-1$ (i.e., $a$ is neither a unit nor trivial).  In that case, the Koszul complex
\[
\xymatrix{0\ar[r] & R\ar[r]^-a & R\ar[r] & S\ar[r] & 0.}
\]
can be chosen as a generator of $\ext^1_R(S,R)$, defining an isomorphism of $S$-modules $\varphi:S\to \ext^1_R(S,R)$. The commutative diagram
\[
\xymatrix{0\ar[r] & R\ar[r]^-a\ar[d]_{-1} & R\ar[r]\ar@{=}[d] & 0 \\
0\ar[r] & R^\vee\ar[r]_-{-a} & R^\vee\ar[r]  & 0}
\]
can be seen as a symmetric quasi-isomorphism and therefore defines an element in $GW^1_0(R)$ that we write $K(a)$. Composing the push-forward map
\[
f_*:GW^0_0(S,\ext^1_R(S,R))\to GW^1_0(R)
\]
with the isomorphism $\chi_a:GW^0(S)\to GW^0(S,\ext^1_R(S,R))$ induced by $\varphi$, we obtain the formula $f_*\chi_a([S,1])=K(a)$ (see for instance \cite[\S 7.2]{Calmes08}).  We may also express the element $K(a)$ as an element of $GW^1_0(R)$ as defined in Section \ref{s:classical}.  Using the results of \cite[\S 7, \S 8]{Walter03} we see that
\[
K(a)=[R\oplus R^\vee, \begin{pmatrix} 0 & -1 \\ \varpi & 0\end{pmatrix}, \begin{pmatrix} 1 \\ 0\end{pmatrix}R,\begin{pmatrix} a \\ 1\end{pmatrix} R^\vee].
\]

If $a_1,\ldots,a_n\in R$ is a regular sequence of elements such that $R/(a_1,\ldots,a_n)$ is smooth, we can iterate the above process and obtain a transfer morphism
\[
f_*:GW^0(S,\ext^n_R(S,R)) \longrightarrow GW^n(R).
\]
Again, there is an isomorphism of $S$-modules $\varphi:S\to \ext^n_R(S,R)$ by mapping $1$ to the Koszul complex associated to the regular sequence $a_1,\ldots,a_n$. Composing $f_*$ with the isomorphism
\[
\chi_{a_1,\ldots,a_n}:GW^0(S) \longrightarrow GW^0(S,\ext^n_R(S,R))
\]
induced by $\varphi$, we obtain $f_*\chi_{a_1,\ldots,a_n}([S,1])=K(a_1,\ldots,a_n)$ where the latter denotes the Koszul complex endowed with the symmetric isomorphism described in \cite[\S 2.4]{Fasel11c}.

\section{Suslin matrices and the degree map}
\label{section:Suslindegree}
The main goal of this section is to establish Theorem \ref{thmintro:matrixgenerator}, which appears here as Theorem \ref{thm:suslinmatrixgwgenerator}.  To establish this theorem, we begin by showing in Subsection \ref{ss:preliminaries} that suitable Grothendieck--Witt groups of odd-dimensional split smooth affine quadrics are free $GW(k)$-modules of rank $1$.  In Subsection \ref{ss:preliminaries} we review Suslin matrices and their basic properties.  In Subsection \ref{ss:degreemap}, we state Theorem \ref{thm:suslinmatrixgwgenerator}, which shows that the Suslin matrices yield explicit generators of the $GW(k)$-modules just mentioned.  Using Suslin's classical $K$-theoretic version of Theorem \ref{thm:suslinmatrixgwgenerator}, Subsection \ref{ss:firstreduction} reduces the proof of Theorem \ref{thm:suslinmatrixgwgenerator} to an explicit computation in Witt groups of split smooth affine quadrics.  Finally, in Subsection \ref{ss:technical}, we perform the (rather long) explicit computation just mentioned to conclude; the argument proceeds by induction and the bulk of the argument amounts to developing techniques to establish the induction step.

\subsection{Preliminaries}
\label{ss:preliminaries}
For any $n\in\N$, let $Q_{2n-1}:=\Spec (k[x_1,\ldots,x_n,y_1,\ldots,y_n]/\sum_{i=1}^n x_iy_i-1)$. Projecting to the first $n$-coordinates yields a map $p_n:Q_{2n-1}\to \A^{n}\setminus 0$ whose fibers are affine spaces of dimension $n-1$. Recall next from \cite[\S 3 Example 2.20]{MV} that we have a weak-equivalence $\A^n\setminus 0\cong S_s^{n-1}\wedge (\gm)^{\wedge n}$ in $\ho k$.  By \cite[Corollary 6.43]{MField}, if $n\geq 2$, there is an isomorphism $[Q_{2n-1},Q_{2n-1}]_{\aone} \cong GW(k)$ and therefore $[Q_{2n-1}, \Omega^{-n}_{\pone}O]_{\aone}$ inherits the structure of $GW(k)$-module by precomposition. If $n=1$, there are isomorphisms
\[
[\gm,(GL/O)_{et}]_{\aone}\cong [\gm,\Omega_{\pone}^4(GL/O)_{et}]_{\aone}\cong[\A^5\setminus 0,(GL/O)_{et}]_{\aone}\cong[Q_9,(GL/O)_{et}]_{\aone},
\]
we equip $[\gm,(GL/O)_{et}]_{\aone}$ with the structure of a $GW(k)$-module via these identifications.

\begin{rem}
The Grothendieck-Witt groups of any (smooth) scheme $X$ are equipped with the structure of $GW(k)$-modules induced by the tensor product and the pullback along the projection $X\to \Spec k$.  Thus, the identification of $[Q_{2n-1}, \Omega^{-n}_{\pone}O]_{\aone}$ as a Grothendieck--Witt group yields another structure of $GW(k)$-module on this group.  We explain why this $GW(k)$-module structure coincides with that described above; we will use this identification freely in the sequel.

Recall that the quadric $Q_{2n-1}$ is $\aone$-weakly equivalent to $S^{n-1} \wedge (\gm)^{\sma n}$.  Using that identification, we have
\[
[Q_{2n-1}, \Omega^{-n}_{\pone}O]_{\aone} \cong [S^{n-1} \wedge (\gm)^{\sma n},\Omega^{-n}_{\pone}O]_{\aone}.
\]
The latter set can be described as $n$-fold contraction of the degree $(n-1)$ $\aone$-homotopy sheaf of $\Omega^{-n}_{\pone}O$, i.e., $(\piaone_{n-1}(\Omega^{-n}_{\pone}O)_{-n})(k)$ (see Theorem \ref{thm:homotopysheavesofgmloopspaces} and the preceding discussion for more details about the notation).  Now, since the sheaf $(\piaone_{n-1}(\Omega^{-n}_{\pone}O)_{-n})$ is a contracted sheaf, its evaluation on $\Spec k$ comes equipped with a $K_0^{MW}(k)$-module structure described in detail in \cite[Lemma 3.49]{MField}.  By the isomorphism $K^{MW}_0(k) \isomt GW(k)$ of \cite{MField}; we obtain another $GW(k)$-module structure on $[Q_{2n-1}, \Omega^{-n}_{\pone}O]_{\aone}$ that we will refer to as the $GW(k)$-module structure induced by contraction.  By construction; this $GW(k)$-module structure agrees with the one arising by precomposition.

That the $GW(k)$-module structures on $[Q_{2n-1}, \Omega^{-n}_{\pone}O]_{\aone}$ induced by contraction and via identification as a Grothendieck--Witt group coincide is established in \cite[Lemma 4.6]{AsokFaselThreefolds} (this result establishes a sheafified version of the required identification).
\end{rem}

\begin{prop}\label{prop:motivatingsuslinmatrices}
For any $n\in\N$, $[Q_{2n-1},\Omega_{\pone}^{-n}O]_{\aone}$ is a free $GW(k)$-module of rank one.
\end{prop}

\begin{proof}
We have a sequence of isomorphisms of $GW(k)$-modules
\[
[Q_{2n-1},\Omega_{\pone}^{-n}O]_{\aone}\cong [\A^n\setminus 0,\Omega_{\pone}^{-n}O]_{\aone}\cong [S_s^{n-1}\wedge(\gm)^{\wedge n},\Omega_{\pone}^{-n}O]_{\aone}\cong [(\pone)^{\wedge (n-1)}\wedge \gm,\Omega_{\pone}^{-n}O]_{\aone}.
\]
By adjunction, the latter is $[\gm, \Omega_{\pone}^{-1}O]_{\aone}=[\gm, (GL/O)_{et}]_{\aone}$.

There is a cofiber sequence of the form
\[
S^0_s \longrightarrow (\gm)_+ \longrightarrow \gm \longrightarrow S^1_s.
\]
Mapping this cofiber sequence into $(GL/O)_{et}$ yields an exact sequence of the form
\[
\widetilde{GW}^1_1(\gm) \longrightarrow GW^1_1(\gm) \longrightarrow GW^1_1(\Spec k).
\]
Since the map $S^0_s \longrightarrow (\gm)_+$ is split by the structure map $(\gm)_+ \to S^0_s$, it follows that $\widetilde{GW}^1_1(\gm) \cong \ker (GW_1^1(k)\to GW_1^1(\gm))$.  The result follows then from \cite[Theorem 9.13]{SchlichtingHKT}.
\end{proof}

The generator of $[\gm, (GL/O)_{et}]$ can be made more explicit. Indeed, we have an obvious map $\gm\to S_2$ given by $t\mapsto \mathrm{diag}(t,-1)$ which induces a map $\gm\to S\cong (GL/O)_{et}$. The class of this map generates $[\gm, (GL/O)_{et}]$ as a $GW(k)$-module by \cite[Remark 9.14]{SchlichtingHKT}. The main purpose of this section is to give an explicit generator of $[Q_{2n-1},\Omega_{\pone}^{-n}O]_{\aone}$ for every $n\in\N$.

\subsection{Suslin matrices}
\label{ss:suslinmatrices}
Suppose $R$ is a commutative unital ring.  In \cite[\S 5]{Suslin77c}, Suslin introduced an inductively defined family of matrices.  We now recall this construction; for related discussion see \cite[Chapter III.7]{Lam}.  Let $a=(a_1,\ldots,a_n)$ and $b=(b_1,\ldots,b_n)$. Define a matrix $\alpha_n(a,b)$ of size $2^{n-1}$ inductively by setting $\alpha_1(a,b)=a_1$ and for $n\geq 2$ setting:
\[
\alpha_n(a,b):=\begin{pmatrix} a_1Id_{2^{n-2}} & \alpha_{n-1}(a^\prime,b^\prime) \\ -\alpha_{n-1}(b^\prime,a^\prime)^t & b_1Id_{2^{n-2}}\end{pmatrix}
\]
where $a^\prime=(a_2,\ldots,a_n)$ and $b^\prime=(b_2,\ldots,b_n)$.

\begin{lem}[{\cite[Lemma 5.1]{Suslin77c}}]\label{lem:alpha}
The matrix $\alpha_n(a,b)$ has the following properties:
\begin{enumerate}[noitemsep,topsep=1pt]
\item $\alpha_n(a,b)\cdot \alpha_n(b,a)^t=(a\cdot b^t) Id_{2^{n-1}}=\alpha_n(b,a)^t\cdot \alpha_n(a,b)$.
\item $\det \alpha_n(a,b)=(a\cdot b^t)^{2^{n-2}}$ if $n\geq 2$.
\end{enumerate}
\end{lem}

In particular, note that if $a\cdot b^t=1$ then $\alpha_n(a,b)\in SL_{2^{n-1}}(R)$ and $\alpha_n(b,a)^t=\alpha_n(a,b)^{-1}$.

Continuing to follow \cite[\S5]{Suslin77c}, we define matrices $I_n$ of size $2^{n-1}$ inductively by setting $I_1=1$ and defining
\[
I_n:=\begin{cases} \begin{pmatrix} 0 & I_{n-1} \\
-I_{n-1} & 0\end{pmatrix} & \text{ if $n=2k$.} \\
\begin{pmatrix} I_{n-1} & 0\\ 0 & -I_{n-1}\end{pmatrix} & \text{ if $n=2k+1$.}
\end{cases}
\]
For any $n\in\N$, the formulas $I_n^t=I_n^{-1}=(-1)^{n(n-1)/2}I_n$ and $\det I_n=1$ hold \cite[Lemma 5.2]{Suslin77c}. As a consequence, $I_n$ is symmetric if $n\equiv 0,1 \pmod 4$ and skew-symmetric if $n\equiv 2,3\pmod 4$.

\begin{lem}[{\cite[Lemma 5.3]{Suslin77c}}]
\label{lem:properties}
The following identities are valid:
\begin{enumerate}[noitemsep,topsep=1pt]
\item $\alpha_n(a,b)\cdot I_n\cdot \alpha_n(a,b)^t=(a\cdot b^t)\cdot I_n$ if $n\equiv 0\pmod 4$,
\item $(\alpha_n(a,b)\cdot I_n)^t=\alpha_n(a,b)\cdot I_n$ if $n\equiv 1\pmod 4$,
\item $\alpha_n(a,b)\cdot I_n\cdot \alpha_n(a,b)^t=(a\cdot b^t)\cdot I_n$ if $n\equiv 2\pmod 4$, and
\item $(\alpha_n(a,b)\cdot I_n)^t=-\alpha_n(a,b)\cdot I_n$ if $n\equiv 3\pmod 4$.
\end{enumerate}
\end{lem}

\subsection{The degree map}
\label{ss:degreemap}
If $R$ is the algebra of functions on $Q_{2n-1}$, then setting $a = (x_1,\ldots,x_n)$ and $b = (y_1,\ldots,y_n)$, we can view $\alpha_n(a,b)$ as a morphism of schemes
\[
\alpha_n: Q_{2n-1} \longrightarrow GL_{2^{n-1}}.
\]
Choosing $(1,0,\ldots,0),(1,0,\ldots,0)$ as a base-point for $Q_{2n-1}$ and the identity for $GL_{2^{n-1}}$, we see that $\alpha_{n}$ is actually a pointed morphism.

Observe that Lemma \ref{lem:properties} shows that $\alpha_n(a,b)\cdot I_n$ is symmetric if $n \equiv 1 \mod 4$ and anti-symmetric if $n \equiv 3 \mod 4$, while $\alpha_n(a,b)$ is orthogonal with respect to the quadratic form defined by $I_n$ if $n \equiv 0 \mod 4$, and symplectic with respect to the symplectic form $I_n$ if $n \equiv 2 \mod 4$.  It follows that the morphisms $\alpha_n$ have image in a subscheme of $GL_{2^{n-1}}$.  We now modify the morphism $\alpha_n$ slightly so that it yields a pointed morphism to the finite dimensional approximations to the spaces $O,(GL/O)_{et},Sp$ or $GL/Sp$ appearing as appropriate loop spaces of $\Z \times OGr$.

Recall the definitions of the matrices $\tau_{2n}, \sigma_{2n}$ and $\psi_{2n}$ from Subsection \ref{ss:degree1}. Define $E_n\in M_{2^{n-1}}(\Z)$ for $n\geq 1$ by setting $E_n=Id_{2^{n-1}}$ if $n=1,2$ and
\[
E_n:=\begin{cases} \begin{pmatrix} 1 & 0\\ 0 & I_{n-1}^t\end{pmatrix}\begin{pmatrix} 1 & 0 \\ \tau_{2^{n-2}} & 1\end{pmatrix} \begin{pmatrix} 1 & -\sigma_{2^{n-2}}\\ 0 & 1\end{pmatrix}\begin{pmatrix} 1 & 0\\ 0 & \psi_{2^{n-2}}\end{pmatrix} & \text{ if $n\equiv 0\pmod 4$.}\\
\begin{pmatrix} E_{n-1} & 0\\ 0 & E_{n-1}\end{pmatrix}\begin{pmatrix} 1 & 0 \\ 0 & \psi_{2^{n-2}}\end{pmatrix} & \text{ if $n\equiv 1\pmod 4$.} \\
\begin{pmatrix} 1 & 0\\ 0 & I_{n-1}^t\end{pmatrix}\begin{pmatrix} 1 & 0 \\ \tau_{2^{n-2}} & 1\end{pmatrix} \begin{pmatrix} 1 & \psi_{2^{n-2}}\\ 0 & 1\end{pmatrix}\begin{pmatrix} 1 & 0\\ 0 & \sigma_{2^{n-2}}\end{pmatrix} & \text{ if $n\equiv 2\pmod 4$.} \\
\begin{pmatrix} E_{n-1} & 0\\ 0 & E_{n-1}\end{pmatrix}\begin{pmatrix} 1 & 0 \\ 0 & \sigma_{2^{n-2}}\end{pmatrix} & \text{ if $n\equiv 3\pmod 4$.}
\end{cases}
\]
Using the identities $I_n^t=I_n^{-1}=(-1)^{n(n-1)/2}I_n$, $\sigma_{2n}^2 = -Id_{2n}$, $\psi_{2n}^2 = -Id_{2n}$ and $\sigma_{2n}\psi_{2n}=-\psi_{2n}\sigma_{2n}$, one obtains the next result by straightforward computation; we leave the details to the reader.

\begin{lem}
\label{lem:correction}
There are equalities of the form:
\[
E_n^t I_n E_n=\begin{cases} \sigma_{2^{n-1}}& \text{ if $n\equiv 0,1\pmod 4$.} \\
\psi_{2^{n-1}} & \text{ if $n\equiv 2,3\pmod 4$.}\end{cases}
\]
\end{lem}

\begin{defn}
\label{defn:psin}
Let $\Psi_1:\gm \to GL_2$ be given by $\Psi_1(t)=\mathrm{diag}(t,-1)$ and define $\Psi_n: Q_{2n-1} \to GL_{2^{n-1}}$ for $n\geq 2$ by the formula
\[
\Psi_n :=\begin{cases}
E_n^{-1}\alpha_n^t E_n, \text{ if } n \text{ is even, and} \\
E_n^t\alpha_n I_nE_n, \text{ if } n \text{ is odd.}
\end{cases}
\]
\end{defn}

As a direct consequence of Lemmas \ref{lem:properties} and \ref{lem:correction}, we obtain the following result.

\begin{prop}
\label{prop:propertiesofpsin}
For $n=1$, the morphism $\Psi_1$ induces a morphism $\gm\to S_2$.  For $n\geq 2$ the morphism $\Psi_n$ induces a pointed morphism
\[
\Psi_n:Q_{2n-1} \longrightarrow \begin{cases} O_{2^{n-1}} & \text{ if $n \equiv 0 \mod 4$}, \\
S_{2^{n-1}}\cong (GL_{2^{n-1}}/O_{2^{n-1}})_{et} & \text{ if $n \equiv 1 \mod 4$}, \\
Sp_{2^{n-1}} & \text{if $n \equiv 2 \mod 4$},  \\
A_{2^{n-1}}\cong GL_{2^{n-1}}/Sp_{2^{n-1}} & \text{if $n \equiv 3 \mod 4$}.
\end{cases}
\]
\end{prop}

\begin{rem}\label{rem:pointedpsi1}
It is easy to make $\Psi_1$ pointed. For this, one has to consider the map $\gm\to S_2$ given by $t\mapsto \begin{pmatrix} t-1 & \frac 12(t+1) \\ \frac 12(t+1) & \frac 14(t-1)\end{pmatrix}$. However, we won't need this property in the sequel.
\end{rem}

\begin{defn}
\label{defn:psinstable}
Abusing notation, we write $\Psi_n$ also for the map
\[
\Psi_n: Q_{2n-1} \longrightarrow \Omega^{-n}_{\pone}O,
\]
obtained by composing the $\Psi_n$ defined above with the stabilization morphism.
\end{defn}

By Proposition \ref{prop:motivatingsuslinmatrices}, the class of $[\Psi_n]$ corresponds to an element of $GW(k)$. The next result, which can be viewed as a Hermitian K-theory analog of \cite[Theorem 2.3]{Suslin82b}, is one of the main results of the paper.

\begin{thm}
\label{thm:suslinmatrixgwgenerator}
For any integer $n \geq 1$, the class $[\Psi_n] \in [Q_{2n-1},\Omega^{-n}_{\pone}O]_{\aone}$ is a generator of this group as a $GW(k)$-module.
\end{thm}

The proof of this result---the technical heart of this paper---is a rather explicit inductive argument, which we distribute over the next two sections.

\subsection{Bott sequences and the first reduction}
\label{ss:firstreduction}
The Suslin matrices $\alpha_n(\alpha,\beta)$ were constructed above as morphisms $Q_{2n-1} \to GL_{2^{n-1}}$ and composing with the stabilization map $GL_{2^{n-1}} \hookrightarrow GL$ we therefore obtain a morphism $Q_{2n-1} \to GL$.  Now, there is a weak equivalence $GL \cong \Omega^1_s (\Z \times Gr)$ by e.g. \cite[Theorem 8.2]{SchlichtingTripathi}, and as a consequence of the Morel--Voevodsky representability theorem \cite[\S 4, Theorem 3.13]{MV} we see that
\[
[Q_{2n-1},GL]_{\aone} \cong [({\pone})^{\sma n},\Z \times Gr]_{\aone} = \tilde{K}_0(({\pone})^{\sma n}) \cong K_0(\Spec k) = \Z.
\]
As mentioned in the previous section, Suslin proved in \cite[Theorem 2.3]{Suslin82b} a variant of Theorem \ref{thm:suslinmatrixgwgenerator} for algebraic K-theory.  More precisely, he established the following result, which we reformulate in our language.

\begin{thm}[Suslin]
\label{thm:suslinmennicke}
If $x = (x_1,\ldots,x_n)$ and $y = (y_1,\ldots,y_n)$, then there is an identification
\[
[Q_{2n-1},GL]_{\aone} \cong \Z [\alpha_n(x,y)]
\]
More generally, for an integer $i \geq 0$, there are identifications
\[
\widetilde{K}_i(Q_{2n-1}) \cong K_{i-1}(\Spec k) \cdot [\alpha_n(x,y)].
\]
\end{thm}

Observe that in particular $[\alpha_n(x,y)^t]$ is an integer multiple of $[\alpha_n(x,y)]$. Since we will need it later, we compute this integer in the next lemma.

\begin{lem}\label{lem:transpose}
For any integer $n \geq 1$, we have $[\alpha_n(x,y)^t] = (-1)^{n+1}[\alpha_n(x,y)]$.
\end{lem}

\begin{proof}
If $n$ is an odd integer, Lemma \ref{lem:properties} yields an equality $\alpha_n(x,y)I_n=(-1)^{n(n-1)/2}I_n^t\alpha_n(x,y)^t$.
The matrices $I_n$ and $I_n^t$ are integer matrices and since $\det I_n=\det I_n^t=1$, it follows that these matrices are elementary.  Thus we see that
\[
[\alpha_n(x,y)]=[(-1)^{n(n-1)/2}\alpha_n(x,y)^t]=[(-1)^{n(n-1)/2}Id]+[\alpha_n(x,y)^t]
\]
in $K_1(Q_{2n-1})$. Since $[(-1)^{n(n-1)/2}Id]$ vanishes in $\widetilde{K}_1(Q_{2n-1})$, the result is proved in that case.

In case $n$ is even, Lemma \ref{lem:properties} states that $\alpha_n(x,y)I_n\alpha_n(x,y)^t=I_n$.  Once more using the fact that $I_n$ is elementary, we obtain the equality $[\alpha_n(x,y)^t]=[\alpha_n(x,y)^{-1}]=-[\alpha_n(x,y)]$.
\end{proof}

Now recall that the Bott sequence (a.k.a. Karoubi periodicity sequence) for (reduced) $GW^i_1$ takes the following form:
\begin{equation}\label{eqn:BottGW1i}
\widetilde{K}_1(Q_{2n-1}) \stackrel{H_{1,n}}{\longrightarrow} \widetilde{GW}^n_1(Q_{2n-1}) \stackrel{\eta}{\longrightarrow} \widetilde{GW}^{n-1}_0(Q_{2n-1}) \stackrel{F_{0,n-1}}{\longrightarrow} \widetilde{K}_0(Q_{2n-1}).
\end{equation}

Theorem \ref{thm:suslinmennicke} gives a description of the first and last terms of this sequence while Proposition \ref{prop:motivatingsuslinmatrices} gives a description of the second term. Moreover, \cite[Lemma 2.4]{Fasel11c} yields an isomorphism $\widetilde{GW}^{n-1}_0(Q_{2n-1})\cong W(k)$ and we therefore deduce the existence of an exact sequence of the form
\[
\Z \cdot [\alpha_n(x,y)] \stackrel{H_{1,n}}{\longrightarrow} GW(k) \stackrel{\eta}{\longrightarrow} W(k) {\longrightarrow} 0.
\]

\begin{prop}\label{prop:GW}
For any integer $n\in\N$, there is an identification
\[
\widetilde{GW}^{n-1}_0(Q_{2n-1})=\begin{cases} W(k)\cdot \eta (E_n^{-1}\alpha_n(x,y)^tE_n) & \text{if $n$ is even}. \\
W(k)\cdot \eta (E_n^t\alpha_n(x,y) I_nE_n) & \text{if $n$ is odd}.\end{cases}
\]
\end{prop}

The proof of this proposition is deferred to the next section.

\begin{proof}[Proof of Theorem \ref{thm:suslinmatrixgwgenerator} assuming Proposition \ref{prop:GW}]
Recall that the Bott sequence \ref{eqn:BottGW1i} reads as
\[
\xymatrix@C=2em{\Z\cdot [\alpha_n(x,y)]\ar[r]^-H & \widetilde{GW}^n_1(Q_{2n-1})\ar[r]^-\eta & \widetilde{GW}^{n-1}_0(Q_{2n-1})\ar[r] & 0.}
\]
and that Proposition \ref{prop:motivatingsuslinmatrices} yields an isomorphism of $GW(k)$-modules $\widetilde{GW}^n_1(Q_{2n-1})=GW(k)\cdot \beta_n$ for some class $\beta_n$.

Suppose first that $n$ is even. Proposition \ref{prop:GW} yields $\eta(\beta_n)=\theta\cdot \eta(E_n^{-1}\alpha_n(x,y)^tE_n)$ for some $\theta\in W(k)$. Let $\Theta$ be a lift of $\theta$ in $GW(k)$ and consider the class of $\beta_n-\Theta\cdot (E_n^{-1}\alpha_n(x,y)^tE_n)$ in $\widetilde{GW}^n_1(Q_{2n-1})$.
The Bott sequence and Lemma \ref{lem:transpose} show that there exists $j\in\N$ such that
\[
\beta_n-\Theta\cdot (E_n^{-1}\alpha_n(x,y)^tE_n)=H(j\cdot [\alpha_n(x,y)^t])=H(j\cdot [E_n^{-1}\alpha_n(x,y)^tE_n]).
\]
Lemma \ref{lem:comparison} shows that $H(j\cdot [E_n^{-1}\alpha_n(x,y)^tE_n])=j\langle 1,-1\rangle\cdot (E_n^{-1}\alpha_n(x,y)^tE_n)$ and it follows that
\[
\beta_n=(\Theta+j\langle 1,-1\rangle)\cdot (E_n^{-1}\alpha_n(x,y)^tE_n)
\]
Therefore $E_n^{-1}\alpha_n(x,y)^tE_n$ is a generator of $\widetilde{GW}^n_1(Q_{2n-1})$ as a $GW(k)$-module.

If $n$ is odd, we replace $E_n^{-1}\alpha_n(x,y)^tE_n$ by $E_n^t\alpha_n(x,y)I_nE_n$ and use the same arguments to finish the proof.
\end{proof}

\subsection{On the Grothendieck-Witt group of $Q_{2n-1}$: proof of Proposition \ref{prop:GW}}
\label{ss:technical}
As seen in Subsection \ref{ss:preliminaries}, we know that the matrix $\mathrm{diag}(t,-1)$ generates $\widetilde{GW}_1^1(\gm)$. Using the description of the periodicity homomorphism $\eta:\widetilde{GW}_1^1(\gm)\to \widetilde{GW}_0^0(\gm)$ given in Subsection \ref{subsec:periodicity}, we see that $\eta(\mathrm{diag}(t,-1))=\langle t\rangle -\langle 1\rangle$, which is a generator of $\widetilde{GW}_0^0(\gm)$.  Thus, Proposition \ref{prop:GW} holds for $n=1$.

To establish the general case, we proceed by induction.  To this end, consider the vanishing locus $Z$ of the global section $x_n$ in $Q_{2n-1}$. It is straightforward to check that $Z \cong Q_{2n-3}\times\A^1$. Consider the following diagram:
\[
\xymatrix{Q_{2n-3}\times \A^1\ar[d]_-{p_{2n-3}}\ar[r]^-{j_{2n-3}} & Q_{2n-1} \\
Q_{2n-3}. & }
\]
By homotopy invariance, $p_{2n-3}$ induces an isomorphism $p_{2n-3}^*$ on reduced Grothendieck--Witt groups.  Regarding the push-forward map $(j_{2n-3})_*$ (described in Section \ref{subsec:Gysin} and choosing the trivialization of the conormal sheaf using the Koszul complex associated to $x_n$), one can establish the following lemma; see \cite[Lemma 2.7]{Fasel11c} for a proof.

\begin{lem}\label{lem:inductionstep}
For $n\geq 2$, the map
\[
(j_{2n-3})_*p_{2n-3}^*:\widetilde{GW}_0^{n-2}(Q_{2n-3}) \longrightarrow \widetilde{GW}_0^{n-1}(Q_{2n-1})
\]
is an isomorphism of $W(k)$-modules.
\end{lem}

Our strategy to prove Proposition \ref{prop:GW} is simply to explicitly compute the push-forward maps.

\begin{lem}
\label{lem:evenpf}
Let $n\geq 2$ be an even integer, $k$ a field with $\mathrm{char}(k)\neq 2$ and $R$ a smooth $k$-algebra.  Suppose $a_1,\ldots,a_n,b_1,\ldots,b_n \in R$ satisfy $\sum a_ib_i=1$.  Set $a:=(a_1,\ldots, a_n)$, $a^\prime:=(a_2,\ldots,a_n)$, $b:=(b_1,\ldots,b_n)$ and $b^\prime:=(b_2,\ldots,b_n)$. If $a_1$ is a regular parameter such that $R/a_1$ is smooth, and
\[
i:GW^{n-2}_0(R/a_1) \longrightarrow GW^{n-1}_0(R)
\]
is the push-forward map obtained from choosing the trivialization of the conormal bundle coming from the Koszul complex associated with $a_1$, then the equality
\[
i(\eta(-E_{n-1}^t\alpha_{n-1}(a^\prime,b^\prime)I_{n-1}E_{n-1}))=\eta(E_n^{-1}\alpha_n(a,b)^tE_n)
\]
holds.
\end{lem}

\begin{proof}
We prove the lemma for $n\equiv 2\mod 4$, the case $n \equiv 0 \mod 4$ being established in a similar fashion.  In this case, $n-1 \equiv 1 \mod 4$, so the matrix $E_{n-1}^t\alpha_{n-1}(a^\prime,b^\prime)I_{n-1}E_{n-1}$ is symmetric by Proposition \ref{prop:propertiesofpsin} and, by Lemma \ref{lem:properties}, the matrix $\alpha_{n-1}(a^\prime,b^\prime)I_{n-1}$ is symmetric.  By the description of the periodicity morphism given in Subsection \ref{subsec:periodicity}, and by Lemma \ref{lem:correction}, the images of $E_{n-1}^t\alpha_{n-1}(a^\prime,b^\prime)I_{n-1}E_{n-1}$ and $\alpha_{n-1}(a^\prime,b^\prime)I_{n-1}$ under the periodicity homomorphism $\eta$ coincide.  To describe $i(\eta(-E_{n-1}^t\alpha_{n-1}(a^\prime,b^\prime)I_{n-1}E_{n-1}))$, we may, equivalently, describe $i(\eta(-\alpha_{n-1}(a^\prime,b^\prime)I_{n-1}))$.  To simplify the notation, we write $\alpha I$ for the matrix $\alpha_{n-1}(a^\prime,b^\prime)I_{n-1}$ in the rest of the proof.

Now, we make the description of the morphism $i$ more explicit using the discussion of Subsection \ref{subsec:Gysin}.  The exact sequence of $R$-modules
\[
\xymatrix@C=3em{0\ar[r] & R^{2^{n-2}}\ar[r]^-{a_1\cdot Id} &  R^{2^{n-2}}\ar[r] &  (R/a_1)^{2^{n-2}}\ar[r] & 0}
\]
and the isomorphism $\chi:R/a_1\to \ext^1_R(R/a_1,R)$ induced by the Koszul complex yield a commutative diagram
\[
\xymatrix@C=3em{0\ar[r] & R^{2^{n-2}}\ar[r]^-{a_1\cdot Id}\ar[d]_-{\alpha I} &  R^{2^{n-2}}\ar[r]\ar[d]^-{-\alpha I} &  (R/a_1)^{2^{n-2}}\ar[r]\ar[d]^-{-\chi \alpha I} & 0 \\
0\ar[r] & (R^{2^{n-2}})^\vee\ar[r]_-{-a_1\cdot Id} &  (R^{2^{n-2}})^\vee\ar[r] &  \ext^1_{R}((R/a_1)^{2^{n-2}}, R)\ar[r] & 0.}
\]
It follows that $i(\eta(-E_n^t\alpha IE_n))$ is the class of the anti-symmetric isomorphism of complexes
\[
\xymatrix@C=4em{0\ar[r] & R^{2^{n-2}}\ar[r]^-{a_1\cdot Id}\ar[d]_-{-\alpha I} &  R^{2^{n-2}}\ar[r]\ar[d]^-{\alpha I }  & 0 \\
0\ar[r] & (R^{2^{n-2}})^\vee\ar[r]_-{-a_1\cdot Id} &  (R^{2^{n-2}})^\vee\ar[r]  & 0}
\]
in $GW^1_0(R)$.

Using the isomorphism of \cite[Theorem 8.1]{Walter03}, we obtain the equality
\[
i([-\alpha I])=\left[R^{2^{n-1}},\begin{pmatrix} 0 & -Id \\ Id & 0\end{pmatrix},\begin{pmatrix} Id \\0 \end{pmatrix} R^{2^{n-2}},\begin{pmatrix} a_1Id \\ -\alpha I\end{pmatrix}R^{2^{n-2}}\right].
\]
If $M$ is the matrix $\mathrm{diag}(Id, -I_{n-1})$, then conjugating the representing matrix of the anti-symmetric form by $M$, moving the Lagrangian subspace by multiplication by the inverse of $M$ and using the formulas $\alpha I=I^t\alpha^t=I^{-1}\alpha^t=I\alpha^t$ from Lemma \ref{lem:properties}, we conclude that there is an equality of the form:
\[
i([-\alpha I])=\left[R^{2^{n-1}},\begin{pmatrix} 0 & I_{n-1} \\ -I_{n-1} & 0\end{pmatrix},\begin{pmatrix} Id \\0 \end{pmatrix} R^{2^{n-2}},\begin{pmatrix} a_1Id \\ \alpha^t\end{pmatrix}R^{2^{n-2}}\right].
\]
Now $E_n^tI_nE_n=\psi_{2^{n-1}}$ by Lemma \ref{lem:correction} and we conclude that there is an equality of the form
\[
i([-\alpha I])=\left[R^{2^{n-1}},\psi_{2^{n-1}},E_n^{-1}\begin{pmatrix} Id \\0 \end{pmatrix} R^{2^{n-2}},E_n^{-1}\begin{pmatrix} a_1Id \\ \alpha^t\end{pmatrix}R^{2^{n-2}}\right].
\]
If $j_n:=E_n^{-1}\begin{pmatrix} Id \\ 0\end{pmatrix}$, we obtain
\[
i([-\alpha I])=\left[R^{2^{n-1}},\psi_{2^{n-1}},j_n(R^{2^{n-2}}),E_n^{-1}\alpha_n(a,b)^tE_n j_n(R^{2^{n-2}})\right].
\]
The right hand side is, by definition, $\eta'(E_n^{-1}\alpha_n(a,b)^tE_n)$ (see the end of Subsection \ref{subsec:periodicity} for the definition of $\eta'$).  Lemma \ref{lem:otherLagrangian} allows us to conclude that $i([\alpha I]) = \eta'(E_n^{-1}\alpha_n(a,b)^tE_n) = \eta(E_n^{-1}\alpha_n(a,b)^tE_n)$, which is precisely what we wanted to establish.
\end{proof}

\begin{lem}
\label{lem:oddpf}
Let $n\geq 3$ be an odd integer, $k$ a field with $\mathrm{char}(k)\neq 2$ and $R$ a smooth $k$-algebra. Let $a_1,\ldots,a_n,b_1,\ldots,b_n \in R$ be such that $\sum a_ib_i=1$.  Set $a:=(a_1,\ldots, a_n)$, $a^\dprime:=(a_3,\ldots,a_n)$, $b:=(b_1,\ldots,b_n)$ and $b^\dprime:=(b_3,\ldots,b_n)$.  If $(a_1,a_2)$ is a regular sequence such that $R/\langle a_1,a_2\rangle$ is smooth, and
\[
i:GW^{n-3}_0(R/\langle a_1,a_2\rangle ) \longrightarrow GW^{n-1}_0(R)
\]
is the pushforward map obtained by choosing the trivialization of the conormal bundle coming from the Koszul complex associated with the regular sequence $(a_1,a_2)$, then the equality
\[
i(\eta(E_{n-2}^t\alpha_{n-2}(a^\dprime,b^\dprime)I_{n-2}E_{n-2}))=\eta(-E_n^t\alpha_n(a,b) I_nE_n)
\]
holds.
\end{lem}

\begin{proof}
We prove the result for $n\equiv 3\mod 4$, since the argument in the other case is proven in an analogous fashion.  Using the same steps as the beginning of the proof of Lemma \ref{lem:evenpf}, we can conclude that the classes of the matrices $E_{n-2}^t\alpha_{n-2}(a^\dprime,b^\dprime)I_{n-2}E_{n-2}$ and $\alpha_{n-2}(a^\dprime,b^\dprime)I_{n-2}$ in $GW_0^0(R)$ coincide, while the classes of $E_n^t\alpha_n(a,b) I_nE_n$ and $\alpha_n(a,b) I_n$ in $GW_0^2(R)$ coincide. To establish the equality of the statement, it therefore suffices to show that
\[
i(\eta(\alpha_{n-2}(a^\dprime,b^\dprime)I_{n-2}))=\eta(-\alpha_n(a,b) I_n).
\]

Observe first that
\[
\alpha_{n-2}(a^\dprime,b^\dprime)I_{n-2}:(R/\langle a_1,a_2\rangle)^{2^{n-3}} \longrightarrow ((R/\langle a_1,a_2\rangle)^{2^{n-3}})^\vee
\]
is a symmetric isomorphism; for the remainder of the proof, to remove some notational clutter, we write $\alpha I$ for $\alpha_{n-2}(a^\dprime,b^\dprime)I_{n-2}$.

There is a projective resolution of $(R/\langle a_1,a_2\rangle)^{2^{n-3}}$ of the form:
\[
\xymatrix@C=3.5em{0\ar[r]  & R^{2^{n-3}}\ar[r]^-{\tiny\mbox{$\begin{pmatrix} a_2Id \\ -a_1Id \end{pmatrix}$}} & R^{2^{n-2}}\ar[r]^*+<1.3em>{\tiny\mbox{$\begin{pmatrix} a_1Id & a_2Id\end{pmatrix}$}} & R^{2^{n-3}}\ar[r] & (R/\langle a_1,a_2\rangle)^{2^{n-3}}\ar[r] & 0}.
\]
This projective resolution shows that the skew-symmetric quasi-isomorphism $\varphi$ defined by:
\[
\xymatrix@C=4em@R=3em{0\ar[r]  & R^{2^{n-3}}\ar[r]^-{\tiny\mbox{$\begin{pmatrix} a_2Id \\ -a_1Id \end{pmatrix}$}}\ar[d]^-{\varphi_2=-\alpha I} & R^{2^{n-2}}\ar[r]^*+<1.3em>{\tiny\mbox{$\begin{pmatrix} a_1Id & a_2Id\end{pmatrix}$}}\ar[d]^-{\varphi_1=\tiny\mbox{$\begin{pmatrix}  0 & \alpha I \\ -\alpha I & 0  \end{pmatrix}$}} & R^{2^{n-3}}\ar[r]\ar[d]^-{\varphi_0=\alpha I} &  0 \\
0\ar[r]  & (R^{2^{n-3}})^\vee\ar[r]_-{\tiny{\begin{pmatrix} a_1Id \\ a_2Id \end{pmatrix} }} & (R^{2^{n-2}})^\vee\ar[r]_-*+<1.3em>{\tiny\mbox{$\begin{pmatrix} a_2Id & -a_1Id\end{pmatrix}$}} & (R^{2^{n-3}})^\vee\ar[r] &  0 }
\]
represents the class of $i(\eta(\alpha I))$ in $GW_0^2(R)$.

The next part of the proof follows the procedure introduced in \cite[\S 4]{Balmer99}.  We first observe that our quasi-isomorphism $\varphi$ above is strongly anti-symmetric in the sense of \cite[Remark 4.3]{Balmer99}.  Consider the following morphism of complexes denoted by $\overline\varphi$:
\[
\xymatrix@C=4em@R=3em{0\ar[r]  & (R^{2^{n-3}})^\vee\ar[r]^-{\tiny{\begin{pmatrix} a_1Id \\ a_2Id \end{pmatrix} }}\ar[d]^-{\overline\varphi_2=-\alpha^\prime I} & (R^{2^{n-2}})^\vee\ar[r]^*+<1.3em>{\tiny\mbox{$\begin{pmatrix} a_2Id & -a_1Id\end{pmatrix}$}}\ar[d]^-{\overline\varphi_1=\tiny\mbox{$\begin{pmatrix}  0 & -\alpha^\prime I \\ \alpha^\prime I & 0  \end{pmatrix}$}} & (R^{2^{n-3}})^\vee\ar[r]\ar[d]^-{\overline\varphi_0=\alpha^\prime I} &  0 \\
0\ar[r]  & R^{2^{n-3}}\ar[r]_-{\tiny\mbox{$\begin{pmatrix} a_2Id \\ -a_1Id \end{pmatrix}$}} & R^{2^{n-2}}\ar[r]_-*+<1.3em>{\tiny\mbox{$\begin{pmatrix} a_1Id & a_2Id\end{pmatrix}$}} & R^{2^{n-3}}\ar[r] &  0 };
\]
here $\alpha^\prime I:=\alpha_{n-2}(b^\dprime,a^\dprime)I_{n-2}=I_{n-2}^t\alpha_{n-2}(b^\dprime,a^\dprime)^t$.
It follows from Lemmas \ref{lem:alpha} and \ref{lem:properties} that $\overline\varphi$ is a quasi inverse of $\varphi$, with a homotopy $\epsilon$ between $Id$ and $\overline\varphi\circ\varphi$
\[
\xymatrix@C=4em@R=3em{0\ar[r]  & R^{2^{n-3}}\ar[r]^-{\tiny\mbox{$\begin{pmatrix} a_2Id \\ -a_1Id \end{pmatrix}$}}\ar[d]_-{\tiny\mbox{$Id-\overline\varphi_2\varphi_2$}} & R^{2^{n-2}}\ar[r]^*+<1.3em>{\tiny\mbox{$\begin{pmatrix} a_1Id & a_2Id\end{pmatrix}$}}\ar[d]_-{\tiny\mbox{$Id-\overline\varphi_1\varphi_1$}}\ar@/_/[ld]_-{\epsilon_1} & R^{2^{n-3}}\ar[r]\ar[d]_-{\tiny\mbox{$Id-\overline\varphi_0\varphi_0$}}\ar@/_/[ld]_-{\epsilon_0} &  0 \\
0\ar[r]  & R^{2^{n-3}}\ar[r]_-{\tiny\mbox{$\begin{pmatrix} a_2Id \\ -a_1Id \end{pmatrix}$}} & R^{2^{n-2}}\ar[r]_*+<1.3em>{\tiny\mbox{$\begin{pmatrix} a_1Id & a_2Id\end{pmatrix}$}} & R^{2^{n-3}}\ar[r] &  0}
\]
given by
\[
\epsilon_0=\begin{pmatrix} b_1Id \\ b_2Id\end{pmatrix}:R^{2^{n-3}} \longrightarrow R^{2^{n-2}}
\]
and
\[
\epsilon_1=\begin{pmatrix} b_2Id & -b_1Id\end{pmatrix}:R^{2^{n-2}} \longrightarrow R^{2^{n-3}}.
\]
We now apply \cite[Definition 4.4]{Balmer99} to obtain a skew-symmetric isomorphism
\[
R^{2^{n-3}}\oplus R^{2^{n-3}}\oplus (R^{2^{n-2}})^\vee \longrightarrow (R^{2^{n-3}})^\vee\oplus (R^{2^{n-3}})^\vee\oplus R^{2^{n-2}}
\]
given explicitly by the matrix
\[
M := \begin{pmatrix} 0 & \alpha I &  a_2Id & -a_1Id \\
-\alpha I & 0 & -b_1Id & -b_2Id \\
-a_2Id & b_1Id & 0 & \alpha^\prime I \\
a_1Id & b_2Id & -\alpha^\prime I & 0\end{pmatrix}.
\]
On the other hand, the matrix $-\alpha_n(a,b)I_n$ is given by the matrix
\[
\begin{pmatrix} 0 & -a_1I_{n-2} & -\alpha_{n-2}(a^\dprime,b^\dprime) I_{n-2} & a_2I_{n-2} \\
a_1I_{n-2} & 0 & -b_2I_{n-2} & -\alpha_{n-2}(b^\dprime,a^\dprime)^t I_{n-2} \\
\alpha_{n-2}(a^\dprime,b^\dprime)I_{n-2} & b_2I_{n-2} & 0 & b_1I_{n-2} \\
-a_2I_{n-2} & \alpha_{n-2}(b^\dprime,a^\dprime)^tI_{n-2} & -b_1I_{n-2} & 0\end{pmatrix}.
\]
Using the formula $I_{n-2}^t=I_{n-2}$, we observe that $-\alpha_n(a,b)I_n$ is obtained from $M$ via conjugation by the matrix
\[
\begin{pmatrix} Id & 0 & 0 & 0 \\
0 & 0 & -Id & 0 \\
0 & 0 & 0 & I_{n-2} \\
0 & I_{n-2} & 0 & 0\end{pmatrix}.
\]
From this, and the explicit description of the periodicity homomorphism from Subsection \ref{subsec:periodicity}, we obtain the equality of the statement.
\end{proof}

\begin{proof}[Proof of Proposition \ref{prop:GW}]
As mentioned at the beginning of this Subsection, the result holds when $n=1$.  Assuming inductively that the result holds for $m\leq n-1$, by Lemma \ref{lem:inductionstep}, there is an isomorphism of $W(k)$-modules
\[
(j_{2n-3})_*p_{2n-3}^*:\widetilde{GW}_0^{n-2}(Q_{2n-3}) \longrightarrow \widetilde{GW}_0^{n-1}(Q_{2n-1})
\]
where $p_{2n-3}:Q_{2n-3}\times \A^1\to Q_{2n-3}$ is the projection and $j_{2n-3}:Q_{2n-3}\times \A^1\to Q_{2n-1}$ is the closed immersion identifying the left-hand term with the vanishing locus of the global section $x_n$.  If $n$ is even, then the result follows immediately from Lemma \ref{lem:evenpf}. If $n\geq 3$ is odd, then consider the following commutative diagram
\[
\xymatrix{Q_{2n-5}\times\A^2\ar[r]^-{i_{2n-5}}\ar[d]_-{q_{2n-5}} & Q_{2n-3}\times \A^1\ar[r]^-{j_{2n-3}}\ar[d]^-{p_{2n-3}} & Q_{2n-1} \\
Q_{2n-5}\times \A^1\ar[r]^-{j_{2n-5}}\ar[d]_-{p_{2n-5}} & Q_{2n-3} &  \\
Q_{2n-5}}
\]
where $q_{2n-5}$ is the projection and $i_{2n-5}$ is the closed immersion associated to the vanishing locus of the global section $x_{n-1}$. Observe that the square in the above diagram is cartesian.  Using the base change formula (which follows essentially from \cite[Theorem 5.2.1]{Calmes09}, we see that $p_{2n-3}^*(j_{2n-5})_*=(i_{2n-5})_*q_{2n-5}^*$ and the result follows now from Lemma \ref{lem:oddpf}.
\end{proof}

\section{Applications}
\label{s:applications}
The goal of this section is to collect some applications of Theorem \ref{thm:suslinmatrixgwgenerator}.  Subsection \ref{ss:natural} recalls the definition of Milnor-Witt K-theory and constructs the ``natural homomorphism from Milnor-Witt K-theory to Grothendieck--Witt groups.  Subsection \ref{ss:suslinmatricesandnaturalhomomorphism} shows that graded components of the natural homomorphism coincide, up to multiplication by a unit in $GW(k)$, with maps induced by $\Psi_n$ (Definition \ref{defn:psin} and Proposition \ref{prop:propertiesofpsin}), which establishes the first part of Theorem \ref{thmintro:suslinmatricesandmatsumoto} from the introduction.  Subsubsection \ref{ss:generalizedmatsumoto} contains the proof of second part of Theorem \ref{thmintro:suslinmatricesandmatsumoto} and Corollary \ref{corintro:matsumoto} from the introduction; these appear here as Theorem \ref{thm:matsumoto} and Corollary \ref{cor:k3ind}.  Finally, Subsection \ref{ss:puncturedaffinespaces} contains the proof of Theorem \ref{thmintro:degreemapindegreen}.

\subsection{Milnor-Witt $K$-theory and the natural homomorphism}
\label{ss:natural}
Let $F$ be a field. Recall from \cite[Definition 5.1]{Morel04} that $K_*^{MW}(F)$ is the free associative $\Z$-graded unital ring generated by symbols $[a]$ with $a\in F^\times$ of degree $1$ and a symbol $\eta$ of degree $-1$ subject to the relations
\begin{enumerate}[noitemsep,topsep=1pt]
\item $[ab]=[a]+[b]+\eta[a][b]$ for any $a,b\in F^\times$.
\item $[a][1-a]=0$ for any $a\neq 0,1$.
\item $\eta[a]=[a]\eta$ for any $a\in F^\times$.
\item $\eta(\eta[-1]+2)=0$.
\end{enumerate}

\begin{lem}
\label{lem:K1MWisom}
For any field $F$ having characteristic unequal to $2$, there is a functorial isomorphism of $K_0^{MW}(F)=GW_0^0(F)$-modules $K^{MW}_1(F) \isomt GW^1_1(F)$.
\end{lem}

\begin{proof}
From \cite[Th\'eor\`eme 5.3]{Morel04}, there is a fiber product presentation of $K^{MW}_1(F)$ (see also \cite[\S 4.1 Remarks and Theorem 5.4]{GilleScullyZhong} for some further discussion).  After modernizing notation a bit, i.e., rewriting $GW^1_1$ in terms of Karoubi's $V$-theory groups, the isomorphism at the level of abelian groups follows from \cite[Corollaire 4.5.1.5]{BargeLannes} or \cite[Lemma 1.2]{HornbostelGersten}.   For the statement regarding module structures, observe that at the level of symbols the isomorphism is given by $[b]\mapsto [F,b,1]$.  The $K_0^{MW}(F)$-module structure on $K^{MW}_1(F)$ is given by the formula $\langle a\rangle [b]=[ab]-[a]$ for any $a,b\in F^\times$ \cite[Lemma 3.5.1]{MField}.  On the other hand, by unwinding the definitions, one sees that the $K^{MW}_0(F) \isomt GW(F)$-module structure on $GW^1_1(F) = V(F)$ is given by the formula $\langle a\rangle \cdot [F,b,1]=[F,ab,a]$ in $V(F)$.  The conclusion of the lemma now follows from the fact that $[F,ab,a]+[F,a,1]=[F,ab,1]$ in $V(F)$.
\end{proof}

Recall that there is a unique element $\eta \in GW^{-1}_{-1}(F)$ corresponding to $\langle 1\rangle\in W(F)$ under the isomorphism $GW^{-1}_{-1}(F) \cong W(F)$.  We write $GW^{*}_*(F)$ for the graded ring $\bigoplus_{n \in \Z} GW^n_n(F)$ with ring structure given by the multiplicative structure in Grothendieck--Witt groups.

\begin{thm}
\label{thm:natural}
Let $F$ be a field of characteristic different from $2$. There is a unique homomorphism of graded rings
\[
\mu: K^{MW}_*(F) \longrightarrow GW^*_*(F)
\]
that (i) in degree $1$ is given by the isomorphism $K^{MW}_1(F) \isomt GW^1_1(F)$, and (ii) sends $\eta \in K^{MW}_{-1}(F)$ to $\eta\in GW^{-1}_{-1}(F)$.  Moreover, this homomorphism is an isomorphism in degree $\leq 2$.
\end{thm}

\begin{rem}
If $\mathbf{KO}$ is the Hermitian K-theory ring spectrum defined, e.g., in \cite{PaninWalterBO}, then there is the unit map $\mathbf{S}^0_k \to \mathbf{KO}$ from the motivic sphere spectrum.  Morel's identification of the zeroth stable homotopy sheaf of the sphere spectrum is an identification $\bpi_{0,j}^{s\aone}(\mathbf{S}^0_k) \cong \K^{MW}_j$.  The ``natural homomorphism" is then the induced map $\bpi_{0,j}^{s\aone}(\mathbf{S}^0_k) \to \bpi_{0,j}^{s\aone}(\mathbf{KO})$.
\end{rem}

\begin{rem}
In unpublished work \cite{SchlichtingMatsumoto}, Schlichting proved Theorem \ref{thm:natural} by a different method. His proof relies on Morel's identification of the graded ring $K^{MW}_*(F)$ as the zeroth stable $\aone$-homotopy sheaf of the sphere spectrum and a proof of the Steinberg relation in stable $\aone$-homotopy theory due to Hu and Kriz \cite{HuKriz}.  Our approach is, in contrast, elementary in the sense that we only use basic ``symbolic manipulations."  We view this approach as more in spirit with the generators and relations definition of the group $K^{MW}_*(F)$.
\end{rem}

\begin{proof}
Since $K^{MW}_*(F)$ is generated by symbols $[a] \in \K^{MW}_1(F)$ and $\eta$ of degree $-1$, if we denote by $s(a)$ the image of $[a]$ under the isomorphism of Lemma \ref{lem:K1MWisom}, then the ring homomorphism should satisfy $[a_1,\ldots,a_n]\mapsto s(a_1)\cdot \ldots\cdot s(a_n)$. In order to prove the theorem, we have to check that the relations in Milnor-Witt $K$-theory are satisfied in $GW^*_*(F)$.

In degree $0$, we set $\langle a \rangle = 1 + \eta[a]$ as in \cite[p. 51]{MField} and abusing notation, we will write $\langle a \rangle$ also for $1 + \eta s(a) \in GW^0_0(F)$.  The relation $[ab] = [a] + [b] + \eta[a][b]$ in Milnor-Witt K-theory can then be rewritten in the form $[ab] = [a] + \langle a \rangle [b]$.  Since, by assumption, the map $K^{MW}_1(F) \to GW^1_1(F)$ is an isomorphism, Lemma \ref{lem:K1MWisom} implies that the same result holds in $GW^1_1(F)$, i.e., $s(ab) = s(a) + \langle a \rangle s(b)$.

The proof is completed in a sequence of steps: the Steinberg relation is checked in Proposition \ref{prop:steinbergrelation}, the relation $\eta(\eta s(-1) + 2) = 0$ is checked in Lemma \ref{lem:hyperbolickilled}, and the relation $\eta s(a) = s(a) \eta$ is checked in Lemma \ref{lem:etacommuteswithsymbols}.  The proofs of Lemmas \ref{lem:hyperbolickilled} and \ref{lem:isomorphisminnegativedegrees} show that $\mu_i$ is an isomorphism if $i \leq 0$.  Finally, the fact that $\mu_2$ is an isomorphism, which is due essentially to Suslin, but observed in this form by Morel, is established in Theorem \ref{thm:mu2isanisomorphism}.
\end{proof}

\begin{lem}
\label{lem:hyperbolickilled}
The relation $\eta(\eta s(-1) +2)=0$ holds.
\end{lem}

\begin{proof}
By \cite[Lemma 3.10]{MField}, the map sending $\langle a \rangle$ to the corresponding $1$-dimensional form $GW(F)$ is an isomorphism.  Since $GW^0_0(F) = GW(F)$, it follows that the induced map $\mu_0: K^{MW}_0(F) \to GW^0_0(F)$ is also an isomorphism.  By \cite[Proposition 6.3]{SchlichtingHKT}, the multiplication by $\eta$ map $GW^0_0(F) \to GW^{-1}_{-1}(F)$ coincides, under the identification of $GW^{-1}_{-1}(F) \cong W(F)$ with the standard surjection $GW(F) \to W(F)$ that kills the ideal generated by the hyperbolic form.  The element $\eta[-1] + 2$ is precisely the class of the hyperbolic form in $GW(F)$, and therefore we see that the relation $\eta (\eta s(-1) + 2) = 0$ holds in $GW^{-1}_{-1}(F)$ as well.
\end{proof}

\begin{lem}
\label{lem:isomorphisminnegativedegrees}
For any $i > 0$, the maps $\mu_{-i}: K^{MW}_{-i}(F) \to GW^{-i}_{-i}(F)$ are isomorphisms.
\end{lem}

\begin{proof}
We saw above that the map $\mu_0$ is an isomorphism, and the map $K^{MW}_0(F) \to K^{MW}_{-1}(F)$ induced by product with $\eta$ coincides with the standard map $GW(F) \to W(F)$ under the identifications mentioned above.  So the map $\mu_1$ is an isomorphism.  By \cite[Proposition 6.3]{SchlichtingHKT} and \cite[Lemma 3.10]{MField}, taking repeated products with $\eta$ we conclude that $\mu_{-i}$ is an isomorphism for $i \geq 2$.
\end{proof}

\begin{lem}
\label{lem:etacommuteswithsymbols}
For any $a\in F^\times$, the relation $\eta\cdot s(a)=s(a)\cdot \eta$ holds.
\end{lem}

\begin{proof}
The product structure on Grothendieck--Witt groups of schemes \cite[\S 9.2]{SchlichtingHKT} satisfies $s(a)\cdot \eta=-\langle -1\rangle \eta\cdot s(a)$. On the other hand, we have $-\langle -1\rangle \eta=\eta(-\langle -1\rangle)=\eta$ since $\eta:GW_0^0(F)\to W(F)$ is the usual map by \cite[Proposition 6.3]{SchlichtingHKT}.
\end{proof}

We now turn our attention to the Steinberg relation.  We begin with some preliminary lemmas.  For any $a\in F^\times$, we denote by $(a)$ the class of $a$ in $K_1(F)$.

\begin{lem}\label{lem:helpful}
Suppose $L$ is a field.  Given elements $a,b,c \in L^\times$, the following equalities hold:
\begin{enumerate}[noitemsep,topsep=1pt]
\item $\langle 1,-1\rangle\cdot  s(a)\cdot s(1-a)=0$;
\item $s(a^2)=\langle 1,-1\rangle s(a)$; and
\item $s(a^2)s(bc)=s(a^2)s(b)+s(a^2)s(c)$.
\end{enumerate}
\end{lem}

\begin{proof}
For Point (1), consider the usual homomorphism $K_2^M(L)\to K_2(L)$, which is an isomorphism by Matsumoto's theorem. Composing with the hyperbolic homomorphism $H_{2,2}:K_2(L)\to GW^2_2(L)$, we obtain a homomorphism
\[
\gamma_2:K_2^M(L) \longrightarrow GW_2^2(L).
\]
If $f_{i,j}:GW_i^j(L)\to K_i(L)$ are the forgetful homomorphisms, there are equalities $f_{1,1}(s(a))=(a)$ and $f_{1,1}(s(b))=(b)$.  Moreover, $f_{2,2}(s(a)\cdot s(b))=f_{1,1}(s(a))\cup f_{1,1}(s(b))=(a)\cup (b)$. Lemma \ref{lem:comparison} yields $H_{2,2}((a)\cup (b))=H_{2,2}f_{2,2}(s(a)\cdot s(b))=\langle 1,-1\rangle\cdot s(a)\cdot s(b)$. The result now follows from the fact that $\{a,1-a\}=0$ in $K_2^M(L)$. \newline

\noindent Point (2) a straightforward consequence of Lemma \ref{lem:K1MWisom} and \cite[Lemma 3.14]{MField}.\newline

\noindent For Point (3), we deduce from Lemma \ref{lem:K1MWisom} and \cite[Lemma 3.14]{MField} that $s(bc)=s(b)+\langle b\rangle s(c)$. Since $s(a^2)=\langle 1,-1\rangle s(a)$ by Point (2) and $\langle b\rangle\langle 1,-1\rangle=\langle 1,-1\rangle$, the result follows.
\end{proof}

If $L/F$ is a finite separable field extension, by the discussion of Subsection \ref{subsec:Gysin}, there is a transfer homomorphism
\[
i_*:GW_i^j(L,\mathrm{Hom}_F(L,F)) \longrightarrow GW_i^j(F).
\]
Choosing the trace form $L\to F$ as a generator of the $L$-vector space $\mathrm{Hom}_F(L,F)$, we obtain a transfer map
\[
i_*:GW_i^j(L) \longrightarrow GW_i^j(F)
\]
that satisfies the usual projection formula. Moreover, there are commutative diagrams of the form
\begin{equation}\label{eqn:Norm}
\xymatrix{GW_1^1(L)\ar[r]^-{i_*}\ar[d]_-{det} & GW_1^1(F)\ar[d]^-{det} \\
L^\times\ar[r]_-{N_{L/k}} & F^\times},
\end{equation}
and
\begin{equation}\label{eqn:trace}
\xymatrix{GW_1^1(L)\ar[r]^-{i_*}\ar[d]_-{det} & GW_1^1(F)\ar[d]^-{det} \\
I(L)\ar[r]_-{Tr} & I(F)},
\end{equation}
where $N_{L/F}$ is the norm map, and $Tr$ is the Scharlau transfer associated with the trace map.

\begin{lem}\label{lem:trace}
If $a\in F^\times$ is not a square, and we set $L=F(\sqrt a)$ and $b=\sqrt a\in L$, then $i_*s(1-b)=\langle 2\rangle s(1-a)$.
\end{lem}

\begin{proof}
By straightforward computation one verifies that $N_{L/F}(1-b)=(1-a)$ and $Tr(\langle-1, 1-b\rangle)=\langle 2\rangle\langle-1, 1-a\rangle$. The result follows from \cite[Th\'eor\`eme 5.3]{Morel04}.
\end{proof}

\begin{prop}
\label{prop:steinbergrelation}
We have $s(a)s(1-a)=0$ in $GW_2^2(F)$.
\end{prop}

\begin{proof}
Suppose first that $a=c^2$ in $F$. In that case, the formula $s(c^2)s(1-c^2)=s(c^2)s(1-c)+s(c^2)s(1+c)$ holds by Lemma \ref{lem:helpful}(3). Now
\[
\langle 1,-1\rangle s(c)=s(c^2)=s((-c)^2)=\langle 1,-1\rangle s(-c)
\]
by Lemma \ref{lem:helpful}(2). It follows that $s(c^2)s(1-c^2)=\langle 1,-1\rangle s(c)s(1-c)+\langle 1,-1\rangle s(-c)s(1+c)$ and we can use Lemma \ref{lem:helpful}(1) to conclude in that case.

Suppose next that $a$ is not a square; set $L=F(\sqrt a)$ and $b=\sqrt a$. Since $F$ has characteristic different from $2$, observe that $L/F$ is separable. By Lemma \ref{lem:helpful}(1), we have $\langle 1,-1\rangle s(b)s(1-b)=0$ in $GW_2^2(L)$. Using Lemma \ref{lem:helpful}(2) once again, it follows that $s(a)s(1-b)=0$.  As above, consider the transfer map $i_*:GW_i^j(L)\to GW_i^j(F)$.  By the projection formula, the equality $i_*(s(a)s(1-b))=s(a)i_*(s(1-b))$ holds.  Lemma \ref{lem:trace} then implies that $\langle 2\rangle s(a)s(1-a)=0$ in $GW_2^2(F)$. Since $\langle 2\rangle$ is invertible by assumption, the result follows.
\end{proof}

\begin{thm}[Suslin]
\label{thm:mu2isanisomorphism}
If $F$ is a field having characteristic unequal to $2$, then the map $\mu_2: K^{MW}_2(F) \to GW^2_2(F) = KSp_2(F)$ is an isomorphism.
\end{thm}

\begin{proof}
This result is actually a slight reformulation of Suslin's result.  The Moore--Matsumoto theorem shows that $KSp_2(F)$ is generated by symbols $[a,b]$, $a,b \in F^{\times}$ satisfying $4$ relations \cite[\S 6]{Suslin87}.  By definition, the map $\mu_2: K^{MW}_2(F) \to KSp_2(F)$ sends the symbol $[a,b]$ to $[a,b]$.  In essence, \cite[Corollaries 6.2, 6.4 and Theorem 6.5 ]{Suslin87} yield an identification of $KSp_2(F)$ with the fiber product of $K_2(F) = K^M_2(F)$ (Matsumoto's theorem) and $I^2(F)$ over $K^M_2(F)/2 \cong I^2(F)/I^3(F)$, where the last isomorphism follows from the Merkurjev--Suslin theorem.  Likewise, Morel shows \cite[Corollaire 5.4]{Morel04} (see the proof of Lemma \ref{lem:K1MWisom} for more comments on this point) that $K^{MW}_2(F)$ admits a presentation as the {\em same} fiber product.  More precisely, $I^3(F) \subset K^{MW}_2(F)$ is generated by elements of the form $-\eta[a,b,c]=-[a]\eta [b,c]=-[a,bc]+[a,b]+[b,c]$, which are sent to the generators Suslin considers.
\end{proof}

\subsection{Suslin matrices and the natural homomorphism}
\label{ss:suslinmatricesandnaturalhomomorphism}
For any $n\in\Z$, let $K_n^{MW}(F)$ be the $n$-th graded piece of $K_*^{MW}(F)$. If $\mathrm{char}(F)\neq 2$, we constructed in the previous section a natural homomorphism $\mu: K^{MW}_*(F) \to GW^*_*(F)$, and we will write $\mu_n$ for the $n$-th graded component of this homomorphism.  Taking the quotient by $\eta$, there is an induced homomorphism of graded rings $K^{MW}_*(F) \to K^M_*(F)$.  Since the natural homomorphism $\mu$ sends $\eta \in K^{MW}_{-1}(k)$ to $\eta \in GW^{-1}_{-1}$ and is compatible with products, it follows that $\mu$ extends the homomorphism $K^M_*(F) \to K^Q_*(F)$ in the sense that the following diagram commutes
\begin{equation}\label{eqn:symbolicmap}
\xymatrix{K_n^{MW}(F)\ar[r]^-{\mu_n}\ar[d] & GW_n^n(F)\ar[d]^-{f_{n,n}} \\
K_n^M(F)\ar[r] & K^Q_n(F)}.
\end{equation}

\begin{rem}
The unit map $\mathbf{S}^0_k \to \mathbf{KGL}$, where $\mathbf{KGL}$ is a ring spectrum representing algebraic K-theory, factors as $\mathbf{S}^0_k \to \mathbf{KO} \to \mathbf{KGL}$ where the second map is the forgetful map, viewed as a morphism of commutative ring spectra.  The above commutative square can also be obtained from this composite by applying $\bpi_{0,j}^{s\aone}(\cdot)$ and observing that $\eta$ is sent to $0$ in algebraic K-theory because $\bpi_{0,j}^{s\aone}(KGL) = 0$ for $j < 0$.
\end{rem}

Morel defines an unramified sheaf $\K_n^{MW}$ whose sections over finitely generated extensions $F/k$ of the base field coincide with $K^{MW}_n(F)$ \cite[\S 3.2]{MField}.  Morel also showed that $\piaone_{n-1}(\A^n\setminus 0)=\K_n^{MW}$ for any $n\geq 2$ \cite[Theorem 6.40]{MField} (since the sphere ${\mathbb A}^n \setminus 0$ is smooth over $\Spec \Z$, the sheaves $\piaone_{n-1}(\A^n \setminus 0)$ are strictly $\aone$-invariant without assumption on $k$).  The $\A^1$-equivalence $Q_{2n-1}\to \A^n\setminus 0$ induces an isomorphism $\piaone_{n-1}(Q_{2n-1})\cong \K_n^{MW}$.

For any $n\geq 2$, the map $\Psi_n:Q_{2n-1}\to \Omega_{\pone}^{-n}O$ defined in Section \ref{section:Suslindegree} induces a map
\[
\Psi_{n,n-1}: \K_n^{MW}\cong \piaone_{n-1}(Q_{2n-1}) \longrightarrow \piaone_{n-1}(\Omega_{\pone}^{-n}O)\cong \mathbf{GW}_n^n
\]
where $\mathbf{GW}_n^n$ is the sheaf associated to $X\mapsto GW_n^n(X)$.

\begin{thm}
\label{thm:multiplication}
If $F$ is a field having characteristic unequal to $2$ and $n\geq 2$, then the homomorphism
\[
\Psi_{n,n-1}(F): \K_n^{MW}(F) \longrightarrow \mathbf{GW}_n^n(F)
\]
coincides (up to a unit in $GW(F)$) with $\mu_n$.
\end{thm}

\begin{proof}
Our proof of this results is modeled on a proof of the corresponding result relating Milnor K-theory and algebraic K-theory established in \cite[Lemma 3.8]{AsokFaselSpheres}. We know that the pointed map $\gm\to (GL_2/O_2)_{et}\to (GL/O)_{et}$ of Remark \ref{rem:pointedpsi1} generates $[\gm,(GL/O)_{et}]$. There is a weak-equivalence $\Omega_s^1(Sp/GL)\cong (GL/O)_{et}$ by \cite[Theorem 8.4]{SchlichtingTripathi} and the above generator corresponds by adjunction to a map
\[
\nu_1:\pone \longrightarrow Sp/GL
\]
which generates $[\pone,Sp/GL]$. The multiplicative structure on Grothendieck--Witt groups yield maps
\[
(Sp/GL)^{\wedge^n} \longrightarrow \Omega_{\pone}^{-n}(\Z\times OGr)
\]
for any $n\in\N$. We then define a map
\[
\nu_n:(\pone)^{\sma n}\stackrel{\nu_1\wedge\ldots\wedge\nu_1}\longrightarrow (Sp/GL)^{\sma^n}\longrightarrow \Omega_{\pone}^{-n}(\Z\times OGr)
\]
for any $n\geq 1$.  The element $\nu_n$ provides a generator of $[(\pone)^{\sma n},\Omega_{\pone}^{-n}(\Z\times Ogr)]=[S^0,\Z\times OGr]=GW(k)$.

Theorem \ref{thm:suslinmatrixgwgenerator} shows that $\Psi_n:Q_{2n-1}\to \Omega_{\pone}^{-n}O$ generates $[Q_{2n-1},\Omega_{\pone}^{-n}O]=[\A^n\setminus 0,\Omega_{\pone}^{-n}O]$.  By means of the identification $O \cong \Omega^1_s (\Z \times OGr)$, the adjoint of $\Psi_n$ also generates then $[(\pone)^{\wedge n},\Omega_{\pone}^{-n}(\Z\times Ogr)]$.  It follows that $\nu_n$ and the adjoint to $\Psi_n$ differ by an invertible element of $GW(k)$.

Now, for every integer $n \geq 1$, by \cite[Theorems 3.37 and 6.40]{MField} there is a canonical ``symbol" morphism $\gm^{\sma n} \to \K^{MW}_n \isomt \bpi_n^{\aone}({\pone}^{\sma n})$.  The induced map on sections over finitely generated extensions $L/k$ assigns to a section $(a_1,\ldots,a_n) \in \gm^{\sma n}(L)$ the symbol $[a_1,\ldots,a_n]$ in $\K^{MW}_n(L)$.  The composite with $\nu_n$ defines a map
\[
\gm^{\sma n} \longrightarrow \K^{MW}_n \stackrel{{\nu_n}_*}\longrightarrow \bpi_n^{\aone}(\Z \times OGr) \cong \mathbf{GW}^n_n.
\]
By construction, this map sends a section $(a_1,\ldots,a_n)$ to $s(a_1)\cdots s(a_n)$.  The result follows by the definition of the natural homomorphism (see the proof of Theorem \ref{thm:natural}).
\end{proof}

\subsection{A generalization of Matsumoto's theorem and $KO_3$}
\label{ss:generalizedmatsumoto}
Theorem \ref{thm:natural} showed that the natural homomorphism $\mu_i$ is an isomorphism in degrees $i \leq 2$.  Using Theorem \ref{thm:multiplication}, we will now demonstrate that $\mu_3$ is also an isomorphism in suitable situations; we view this is a generalization of Matsumoto's theorem.

\begin{thm}
\label{thm:matsumoto}
If $F$ is a field having characteristic different from $2$, then the homomorphism $\mu_3:K_3^{MW}(F)\to GW_3^3(F)$ is an isomorphism.
\end{thm}

\begin{proof}
Observe first that the diagram
\[
\xymatrix{SL_2\ar[r]\ar[d] & Sp_4\ar[d] \\
SL_3\ar[r] & SL_4,}
\]
where the top horizontal map is defined (functorially) by sending $M\in SL_2$ to the block-diagonal matrix $diag(Id_2,M)$, the bottom horizontal map sends $M\in SL_3$ to the block-diagonal matrix $diag(1,M)$, the left-hand map is determined by the fact that the composite $SL_2\to Sp_4\to SL_4$ factors through $SL_3$ and the right-hand vertical map is the standard inclusion, is Cartesian.

Observe that sending a matrix $M\in SL_3$ to its first row and the first column of its inverse yields an isomorphism $SL_3/SL_2\to Q_5$. Similarly, we obtain an isomorphism $SL_4/SL_3\cong Q_7$.  Since $Sp_4$ acts transitively on $SL_4/SL_3$, and the stabilizer of the identity coset is $SL_2$, we conclude that the diagram induces an isomorphism of schemes $Sp_4/SL_2 \cong SL_4/SL_3$ and consequently an isomorphism $\beta:Q_5\cong SL_3/SL_2\to SL_4/Sp_4$.

Composing $\beta$ with the stabilization map $SL_4/Sp_4\to GL_4/Sp_4\to GL/Sp$, we obtain a a map $\beta:Q_5\to GL/Sp$. Let $\tau:Q_5\to Q_5$ be the isomorphism defined by $x_2\mapsto -x_2$ and $y_2\mapsto -y_2$ and fixing the other coordinates. We claim that the diagram
\begin{equation}\label{eqn:psibeta}
\xymatrix{Q_5\ar[r]^-{\Psi_3}\ar[d]_-\tau & GL_4/Sp_4\ar@{=}[d] \\
Q_5\ar[r]_-{\beta} & GL_4/Sp_4}
\end{equation}
commutes.  To see this, it suffices to check that the maps described define an isomorphism of Zariski sheaves.  If $R$ is an arbitrary local $k$-algebra, then any element of $Q_5(R)$ may be written in the form $(e_1G,G^{-1}e_1^t)$ for a suitable matrix $G$ in $SL_3(R)$.  With that in mind, observe that $\beta$ is pointed and that $\beta(e_1G,G^{-1}e_1^t)= \mathrm{diag}(1,G)$ for any $G\in SL_3(R)$. It follows that $\beta(e_1G,G^{-1}e_1^t)=\mathrm{diag}(1,G^t)\psi_4\mathrm{diag}(1,G)$ under the isomorphism $GL_4/Sp_4\simeq A_4$. Explicitly, if $(a,b,c)=e_1G$ and $(a^\prime,b^\prime,c^\prime)=G^{-1}e_1^t$, we get
\[
\beta(e_1G)=\begin{pmatrix} 0 & a & b & c \\
-a & 0 & c^\prime & -b^\prime \\
-b & -c^\prime & 0 & a^\prime \\
-c & b^\prime & -a^\prime & 0\end{pmatrix}
\]
(compare with \cite[proof of Theorem 5.2 (a)]{VasersteinSuslin}).

On the other hand,
\[
\Psi_3(e_1G,G^{-1}e_1^t)=\begin{pmatrix} 0 & a & -b & c \\
-a & 0 & c^\prime & b^\prime \\
b & -c^\prime & 0 & a^\prime \\
-c & -b^\prime & -a^\prime & 0\end{pmatrix}
\]
and the claim follows. As a consequence, $\Psi_3$ induces an isomorphism $Q_5\to SL_4/Sp_4$ and therefore an isomorphism $\K_3^{MW}\to \mathbf{GW}_3^3$ by \cite[Proposition 4.2.2]{AsokFaselpi3a3minus0} (again, the schemes $SL_{2n}/Sp_{2n}$ are smooth over $\Spec \Z$ and, as mentioned in the Preliminaries, their $\aone$-homotopy sheaves are strictly $\aone$-invariant over an arbitrary base field). The result follows then from Theorem \ref{thm:multiplication}.
\end{proof}

\begin{rem}
Observe that the conclusion of Theorem \ref{thm:matsumoto} contradicts the expectations mentioned in \cite[Remark 30]{Fasel09c}.
\end{rem}

\begin{ex}
If $F$ is a field, the natural homomorphism $K^M_3(F) \to K^Q_3(F)$ is known to be injective \cite[Chapter VI Proposition 4.3.2]{KBook} but not surjective in general.  In contrast, the natural homomorphism $K^M_4(\real) \to K^Q_4(\real)$ is known not be injective.  Indeed, the class of $\{-1,-1,-1,-1\} \in K^M_4(\real)$ is non-zero (and $2$-torsion), but the image of this class in $K^Q_4(\real)$ is zero \cite[Chapter IV Exercise 1.12]{KBook}.  Here, we analyze this example in the context of the natural homomorphism $\mu$.  By theorem \ref{thm:natural}, since $\eta$ in $K^{MW}_{-1}(F)$ is sent to $\eta \in GW^{-1}_{-1}(F)$, there is a commutative diagram of the form
\[
\xymatrix{
K^{MW}_4(\real) \ar[r]^{\mu_4}\ar[d]^{\cdot \eta} & GW^4_4(\real) \ar[d]^{\cdot \eta} \\
K^{MW}_3(\real) \ar[r]^{\mu_3} & GW^3_3(\real).
}
\]
The map $\mu_3$ is an isomorphism by Theorem \ref{thm:matsumoto}, while the image of the left vertical arrow coincides with $I^4(\real)$ by the fiber product description of $K^{MW}_4(\real)$ \cite[Th{\'e}or{\`e}me 5.3]{Morel04}.  One knows that $I^4(\real) \cong \Z$ by Sylvester's inertia theorem, and under the left vertical map, the element $[-1,-1,-1,-1]$ is sent to a non-zero element of $I^4(\real)$ (specifically, the Pfister form $\langle \langle -1,-1,-1,-1 \rangle \rangle$) and is thus non-torsion.  By means of the isomorphism $\mu_3$, the element $\eta[-1,-1,-1,-1]$ is sent to a non-torsion element of $GW^3_3(\real)$.  Therefore, $\mu_4([-1,-1,-1,-1])$ is a non-torsion element of $GW^4_4(\real)$.  A similar argument shows that $n$-fold product of $[-1] \in K^{MW}_1(\real)$ with itself is sent to a non-torsion element of $GW^n_n(\real)$ under $\mu_n$.  We do not know whether the map $K^{MW}_4(F) \to GW^4_4(F)$ is injective in general, though one can show this is so for many fields including $\real$.
\end{ex}

Recall that for any field $F$, the group $K_3^{ind}(F)$ is the cokernel of the symbol map $K_3^M(F)\to K_3(F)$; this is the "non-symbolic" part of $K_3(F)$ and is rather difficult to understand in general.  Theorem \ref{thm:matsumoto} has the following amusing consequence.

\begin{cor}
\label{cor:k3ind}
If $F$ is a field having characteristic unequal to $2$, then the hyperbolic map $H_{3,0}:K_3(F)\to GW_3^0(F)$ induces an isomorphism $K_3^{ind}(F) \isomt KO_3(F)$.
\end{cor}

\begin{proof}
Recall from Subsection \ref{ss:modern} that there are canonical isomorphisms $GW_3^0(F) \isomt KO_3(F)$.  Completing the diagram \ref{eqn:symbolicmap} with the Bott sequence, we obtain a commutative diagram
\[
\xymatrix{K_3^{MW}(F)\ar[r]\ar[d]_-{\mu_3} & K_3^M(F)\ar[r]\ar[d]  & 0 & \\
GW_3^3(F)\ar[r]_-{f_{3,3}} & K_3(F)\ar[r]_-{H_{3,0}} & GW_3^0(F)\ar[r]_-\eta & GW_2^3(F)}
\]
with exact rows. Since $\mu_3$ is an isomorphism, there is an associated exact sequence of the form:
\[
\xymatrix{0\ar[r] & K_3^{ind}(F)\ar[r]^-{H_{3,0}} & GW_3^0(F)\ar[r]_-\eta & GW_2^3(F).}
\]
The vanishing of $GW^3_2(F)$ is established in \cite[Lemma 2.2]{Fasel10b}, and we thus obtain the claimed isomorphism.
\end{proof}

\subsection{Applications to the structure of $\pi_n^{\aone}({\mathbb A}^n \setminus 0)$}
\label{ss:puncturedaffinespaces}
Recall that if $\mathbf{F}$ is a presheaf on $\Sm_k$, then its contraction $\mathbf{F}_{-1}$ is the presheaf defined by the formula:
\[
\mathbf{F}_{-1}(X) := \ker(\mathbf{F}(X\times \gm) \longrightarrow \mathbf{F}(X))
\]
induced by the unit map $X\to X\times \gm$. The $n$-th contraction $\mathbf{F}_{-n}$ of $\mathbf{F}$ is inductively defined by $\mathbf{F}_{-n}:=(\mathbf{F}_{1-n})_{-1}$.  If $\mathbf{F}$ is a sheaf, then so is $\mathbf{F}_{-1}$. Furthermore, contraction is an exact functor on the category of (strongly or) strictly $\aone$-invariant sheaves, e.g., by \cite[Lemma 7.33]{MField}.  The following result of Morel explains one reason why the contraction construction is important.

\begin{thm}[{\cite[Theorem 6.13]{MField}}]
\label{thm:homotopysheavesofgmloopspaces}
If $(\mathscr{X},x)$ is a pointed $\aone$-connected space, then for every pair of integers $i,j \geq 1$
\[
\bpi^{\aone}_{i,j}( \mathscr{X}):= \bpi_i^{\aone}({\mathbf R}\Omega^j_{\gm{}} \mathscr{X}) = \bpi_i^{\aone}(\mathscr{X})_{-j}.
\]
\end{thm}

\begin{rem}
Contracted sheaves appear in the Gersten resolution of a strictly $\aone$-invariant sheaf as explained in \cite[Chapter 5]{MField}.
\end{rem}

The Grothendieck--Witt sheaves $\mathbf{GW}_i^j$ are the Nisnevich sheaves associated with the presheaves $X\mapsto GW_i^j(X)$; as mentioned in \cite[\S 4]{AsokFaselThreefolds}, these sheaves are strictly $\aone$-invariant.  The formula $(\mathbf{GW}^j_i)_{-1} \cong \mathbf{GW}^{j-1}_{i-1}$ is established in \cite[Proposition 4.4]{AsokFaselThreefolds}.

Consider the map $\Psi_n:Q_{2n-1}\to \Omega_{\pone}^{-n}O$ of Definition \ref{defn:psinstable}.  Applying $\piaone_{n,j}(\_)$, there are induced maps
\[
\Psi_{n,n,j}:\piaone_{n,j}(\A^n\setminus 0)\cong \piaone_n(Q_{2n-1}) \longrightarrow \piaone_{n,j}(\Omega_{\pone}^{-n}O)\cong \mathbf{GW}_{n+1-j}^{n-j}.
\]
In particular, when $j = 0$, there is a morphism
\[
\Psi_{n,n}: \bpi_n^{\aone}({\mathbb A}^n \setminus 0) \longrightarrow \mathbf{GW}^n_{n+1}.
\]
The maps so defined have already been studied for $n = 2,3$.

\begin{ex}
When $n = 2$, the map $\Psi_2$ is the map $Q_{3} \to SL_2$ sending $(a_1,a_2,b_1,b_2)$ satisfying $a_1b_1 + a_2b_2 = 1$ to the matrix
\[
\begin{pmatrix}
a_1 & a_2 \\
-b_2 & b_1
\end{pmatrix}
\]
composed with stabilization.  In other words, $\Psi_2$ is precisely the pointed map $Sp_2 \to Sp_{\infty}$ under the identification $Sp_2 = SL_2$.  Thus, the map $\Psi_{2,2}$ coincides, by construction, with the (surjective) morphism $\bpi_2^{\aone}({\mathbb A}^2 \setminus 0) \to \mathbf{GW}^2_3$ studied in \cite[Theorem 3.3]{AsokFaselThreefolds}.
\end{ex}

\begin{ex}
The morphism $\Psi_{3,3}$ coincides with the (surjective) morphism $\bpi_3^{\aone}({\mathbb A}^3 \setminus 0) \to \mathbf{GW}^3_4$ constructed in \cite[Proposition 4.2.2]{AsokFaselpi3a3minus0} as discussed in Subsection \ref{ss:generalizedmatsumoto}.
\end{ex}

The next theorem shows that the image of $\Psi_{n,n}$ is always non-trivial; in particular, it provides interesting elements of $\piaone_{n,n-3}(\A^n\setminus 0)$.

\begin{thm}
\label{thm:surjectivecontraction}
For $n\geq 4$, there is an equality $(\Psi_{n,n})_{3-n} = \Psi_{n,n,n-3}$, and
\[
(\Psi_{n,n})_{3-n}:\piaone_{n,n-3}(\A^n\setminus 0) \longrightarrow (\mathbf{GW}_{n+1}^n)_{3-n}\cong \mathbf{GW}_4^3
\]
is surjective.
\end{thm}

\begin{proof}
The equality of the statement is a consequence of adjunction.  By periodicity, we can see $\Psi_3:\A^3\setminus 0\to \Omega_{\pone}^{-3}O$ as a map $\A^3\setminus 0\to \Omega_{\pone}^{n-3}\Omega_{\pone}^{-n}O$ that we still denote by $\Psi_3$. Its adjoint $\tilde\Psi_3:\A^n\setminus 0\to \Omega_{\pone}^{-n}O$ generates $[\A^n\setminus 0,\Omega_{\pone}^{-n}O]\cong [\A^3\setminus 0,\Omega_{\pone}^{n-3}\Omega_{\pone}^{-n}O]$ and it follows from Theorem \ref{thm:suslinmatrixgwgenerator} that $\tilde\Psi_3$ coincides with $\Psi_n$ (up to some invertible element of $GW(k)=[\A^n\setminus 0,\A^n\setminus 0]$). By adjunction again, the following diagram is therefore commutative
\[
\xymatrix{\Omega_{\pone}^{n-3}(\A^n\setminus 0)\ar[r]^-{\Psi_n} & \Omega_{\pone}^{n-3}\Omega_{\pone}^{-n}O \\
\A^3\setminus 0\ar[u]\ar[ru]_-{\Psi_3} & }
\]
where $\A^3\setminus 0\to \Omega_{\pone}^{n-3}\Sigma_{\pone}^{n-3}(\A^3\setminus 0)\cong \Omega_{\pone}^{n-3}(\A^n\setminus 0)$ is the unit map. The result follows then from \cite[Theorem 4.4.1]{AsokFaselpi3a3minus0} (together with Diagram (\ref{eqn:psibeta})).
\end{proof}

\begin{rem}
We believe that $(\psi_{n,n})_{-i}:\piaone_{n,i}(\A^n\setminus 0)\to \mathbf{GW}_{n+1-i}^{n-i}$ is not surjective in general for $n\geq 4$ and $i\leq n-4$.  A suitable stabilized version of the results above are established in \cite{RondigsSpitzweckOstvaer}.
\end{rem}

In view of Theorem \ref{thm:surjectivecontraction}, the kernel of $\Psi_{n,n}$ is a very interesting sheaf. Let $\mathcal F_n$ be the homotopy fiber of $\Psi_n:Q_{2n-1}\to \Omega_{\pone}^{-n}O$.  As a corollary of the theorem, we obtain the following connectivity result.

\begin{cor}
For any $n\geq 3$, the space $\Omega_{\pone}^{n-3}\mathcal F_n$ is $\aone-2$-connected.
\end{cor}

\begin{proof}
Applying the functor $\Omega_{\pone}^{n-3}$ to the $\aone$-fiber sequence ${\mathcal F}_n \to \A^n\setminus 0 \to \Omega_{\pone}^{-n}O$ produces a fiber sequence of the form:
\[
\Omega_{\pone}^{n-3}\mathcal F_n \longrightarrow \Omega_{\pone}^{n-3}(\A^n\setminus 0) \longrightarrow  \Omega_{\pone}^{n-3}\Omega_{\pone}^{-n}O\cong \Omega_{\pone}^{-3}O\cong GL/Sp.
\]
Now, apply Theorems \ref{thm:surjectivecontraction} and \ref{thm:matsumoto}.
\end{proof}


\begin{footnotesize}
\bibliographystyle{alpha}
\bibliography{KODegree}
\end{footnotesize}
\Addresses
\end{document}